\documentclass[11pt]{amsart}
\usepackage{amsmath, amssymb}
\usepackage[alphabetic]{amsrefs}
\usepackage{amsfonts, epsfig,color,wrapfig,epstopdf}
\usepackage{ xypic,verbatim,amscd,color}
\usepackage{youngtab, ytableau}
\usepackage{graphicx}
\usepackage{colordvi}

\numberwithin{equation}{section}

\newcommand\op{\operatorname}

\newcommand\bull{\sssize{\bullet}}

\newcommand\yiy{\widetilde{D}}
\newcommand\aresim{\widehat{R}}

\newcommand\frh{\mathfrak{h}}

\newcommand\killing{\kappa}

\newcommand{\letof}[1]{\stackrel{#1}{\longrightarrow}}
\newcommand\mf{\mathcal{F}}
\newcommand\yosim{\widehat{Y}^0}

\newcommand\Pic{\operatorname{Pic}}

\newcommand\Fl{\operatorname{Fl}}

\newcommand\tensor{\otimes}
\newcommand\ml{\mathcal{L}}

\newcommand\mb{\mathcal{B}}

\newcommand\GL{\operatorname{GL}}

\newcommand\ms{\mathcal{S}}

\newcommand\Gr{\operatorname{Gr}}
\newcommand\SL{\operatorname{SL}}

\newcommand{\leto}[1]{\stackrel{#1}{\to}}

\newcommand{\mfq}{\mathcal{F}_{\Bbb{Q}}}
\newcommand{\mfqtwo}{\mathcal{F}_{2,\Bbb{Q}}}

\newcommand\gammanq{\Gamma_{n,\Bbb{Q}}(s)}

\newcommand\SLL{\operatorname{SL}}

\newtheorem{theorem}{Theorem}[section]

\newtheorem{remark}[theorem]{ Remark}

\newtheorem{corollary}[theorem]{Corollary}

\newtheorem{proposition}[theorem]{Proposition}

\newtheorem{problem}[theorem]{Problem}
\newtheorem{lemma}[theorem]{Lemma}

\newtheorem{definition/lemma}[theorem]{Definition/Lemma}
\newtheorem{defi}[theorem]{Definition}

\newtheorem{notation}[theorem]{Notation}

\newtheorem{example}[theorem]{\bf Example}

\begin{document}
\title[Extremal rays of eigencones]{Extremal rays in the Hermitian eigenvalue problem}
\author{Prakash Belkale}

\maketitle
\begin{abstract}
The Hermitian eigenvalue problem asks for  the possible eigenvalues of a sum of $n\times n$ Hermitian matrices, given the eigenvalues of the summands.
The regular faces of the cones $\Gamma_n(s)$ controlling this problem  have been characterized in terms of classical Schubert calculus by the work of several authors.

We determine extremal rays of $\Gamma_n(s)$ (which are never regular faces) by relating them to the geometry of flag varieties: The extremal rays either arise from ``modular intersection loci", or by ``induction'' from extremal rays of smaller groups. Explicit formulas are given for both the extremal rays coming from such intersection loci, and for the induction maps.
\end{abstract}
\section{Introduction}
The classical Hermitian eigenvalue problem asks for the possible eigenvalues of a sum of Hermitian matrices given the eigenvalues of the summands (see the survey \cite{FBulletin})

We first introduce notation for the Hermitian eigenvalue problem.
\begin{defi}\label{nafnew}
Define the  polyhedral cone
\begin{equation}\label{weylC}
\frh_{+,n}=\{{x}=(x^{(1)},\dots,x^{(n)})\in \Bbb{R}^n \mid  \sum_{a=1}^n x^{(a)}=0,\ x^{(1)}\geq x^{(2)}\geq\dots\geq x^{(n)}\}.
\end{equation}
Let $s\geq 3$ be a fixed integer. Define the eigencone $\Gamma_n(s)\subset (\frh_{+,n})^s$ to the set of $s$-tuples $(x_1,\dots,x_s)$ such that there exist traceless $n\times n$ Hermitian matrices $A_1,\dots,A_s$, such that
\begin{enumerate}
\item For all $1\leq i\leq s$, the eigenvalues of $A_i$ are the numbers $x^{(a)}_i$, $a=1,\dots,n$. Here $x_i=(x_i^{(1)},\dots,x_i^{(n)})$.
\item $\sum_{i=1}^s A_i=0$.
\end{enumerate}
$\Gamma_n(s)$ is well known to be a rational polyhedral cone, see e.g., \cite{FBulletin}. The walls of $\Gamma_n(s)$ obtained by intersecting it with walls of $\frh_{+,n}^s$ are called the (Weyl) chamber walls of $\Gamma_n(s)$.
Note the standard fact that the space of traceless Hermitian matrices is identified with the Lie algebra of the special unitary group $\op{SU}(n)$.
\end{defi}

We consider the following problem
\begin{problem}\label{mainp}
Find the extremal rays of the cone $\Gamma_n(s)$.
\end{problem}
It was shown by Klyachko \cite{Kly} that the cone $\Gamma_n(s)$  is defined inside the set of triples $(x_1,\dots,x_s)$ by a system of inequalities controlled by the Schubert calculus of Grassmannians $\Gr(r,n)$. In \cite{BLocal}, the  intersection number one Klyachko inequalities were shown to be sufficient for cutting out the polyhedral cone $\Gamma_n(s)$. This smaller set of inequalities was then shown to give an irredundant set by Knutson, Tao and Woodward \cite{KTW}, amounting to a characterization of the regular facets (i.e., codimension one faces which are not  contained in  (Weyl) chamber walls in one of the coordinates, see Definition \ref{nafnew}) of $\Gamma_n(s)$. Higher codimension regular faces (i.e., not contained in a Weyl chamber wall) have been characterized by Ressayre \cite{R2} in terms of the deformed cup product on the cohomology of flag varieties introduced by the author and S. Kumar in \cite{BK}. Extremal rays, i.e., faces of dimension one, of $\Gamma_n(s)$, are on Weyl chamber walls  (see Lemma \ref{faceto}), are therefore not regular faces of $\Gamma_n(s)$, and  the above results on regular faces do not give  information about the extremal rays of $\Gamma_n(s)$.
The problem above is related to a problem of invariants in tensor products, which we now describe;
\subsection{Invariants in tensor products}
Define the rational Weyl chamber,
\begin{equation}\label{weylCQ1}
\frh_{+,n,\Bbb{Q}}= \{{x}=(x^{(1)},\dots,x^{(n)})\in \Bbb{Q}^n \mid  \sum_{a=1}^n x^{(a)}=0,\  x^{(1)}\geq x^{(2)}\geq \dots \geq x^{(n)}\}
\end{equation}
and the rational Cartan vector space which contains \eqref{weylCQ1}
\begin{equation}\label{rat1}
\frh_{n,\Bbb{Q}}= \{{x}=(x^{(1)},\dots,x^{(n)})\in \Bbb{Q}^n \mid  \sum_{a=1}^n x^{(a)}=0\}.
\end{equation}

Recall that irreducible representations of $\operatorname{GL}(n)$ are parameterized by sequences
$\lambda=(\lambda^{(1)},\dots,\lambda^{(n)})\in\Bbb{Z}^n$ with  $\lambda^{(1)}\geq \lambda^{(2)}\geq \dots\geq \lambda^{(n)}$. We denote the irreducible representation corresponding to $\lambda$ by $V_{\lambda}$. Two representations $V_{\lambda}$ and $V_{\mu}$ restrict to the same representation of $\SLL(n)$ if and only if $\lambda=\mu+c(1,\dots,1)$ for some integer $c$. The set of irreducible representations of $\SLL(n)$ is therefore in one-one correspondence with the set of $\lambda$ as above with $\lambda_n=0$, which in turn is in one-one correspondence with the set of dominant integral weights in $\frh_{n,\Bbb{Q}}^*$.

Recall the Killing form isomorphism $\frh_{n,\Bbb{Q}}^*=\frh_{n,\Bbb{Q}}$. This takes a $\lambda\in \frh^*_{n,\Bbb{Q}}$ as above to
a point ${\killing}(\lambda)=(x^{(1)},\dots,x^{(n)})\in \frh_{n,\Bbb{Q}}$  by the formulas
\begin{equation}\label{formule1}
{\killing}(\lambda)=(x^{(1)},\dots,x^{(n)})\in \frh_{n,\Bbb{Q}},\ \  x^{(a)}=\lambda^{(a)}-\frac{|\lambda|}{n},\ a=1,\dots,n,\  |\lambda|=\sum_{a=1}^n \lambda^{(a)}.
\end{equation}
Note that $V_{\lambda}$ and $V_{\mu}$ restrict to the same representation of $\SLL(n)$ if and only if  $\killing(\lambda)=\killing(\mu)$.

Let $\operatorname{Fl}(n)=\Fl(\Bbb{C}^n)$ denote the complete flag variety parameterizing
complete flags of vector spaces
$F_{\bullet}: 0\subsetneq F_1\subsetneq F_2\subsetneq \dots \subsetneq F_{n}=\Bbb{C}^n$ in $\Bbb{C}^n$. For each irreducible representation
$V_{\lambda}$ of $\operatorname{GL}(n)$, there is a $\operatorname{GL}(n)$-equivariant line bundle $L_{\lambda}$ on $\operatorname{Fl}(n)$
such that $H^0(\operatorname{Fl}(n),L_{\lambda})=V_{\lambda}^*$ as representations of $\operatorname{GL}(n)$ (see Definition \ref{line}). Recall that $\Pic(\Fl(n))\tensor\Bbb{Q}=\frh_{n,\Bbb{Q}}^*$, and all line bundles on $\Fl(n)$ are canonically $\SLL(n)$ linearized.
\begin{remark}\label{identify}
We use \eqref{formule1} to identify
$\frh_{n,\Bbb{Q}}$ with the rational Picard group $\Pic_{\Bbb{Q}}(\Fl(n))=\Pic(\Fl(n))\tensor \Bbb{Q}$. This restricts  to an identification of $\frh^{+}_{n,\Bbb{Q}}$
and $\Pic^{+}_{\Bbb{Q}}(\Fl(n))$ defined as the $\Bbb{Q}$-span of effective line bundles.
\end{remark}

The following well known result (see e.g., \cite{FBulletin}) gives a characterization of the eigencone in terms of invariants of tensor products:
\begin{proposition}\label{wellknown1}
Let $V_{\lambda_1},V_{\lambda_2},\dots, V_{\lambda_s}$ be irreducible representations of $\SLL(n)$.
The following are equivalent:
\begin{enumerate}
\item $(V_{N\lambda_1}\tensor V_{N\lambda_2}\tensor\dots\tensor V_{N\lambda_s})^{\SLL(n)}\neq 0$ for some positive integer $N$,
\item $H^0(\operatorname{Fl}(n)^s , L^N)^{\SLL(n)} \neq 0$ for some positive integer $N>0$ where $L=L_{\lambda_1}\boxtimes L_{\lambda_2}\boxtimes\dots\boxtimes L_{\lambda_s}$.
\item $(\killing(\lambda_1),\killing(\lambda_2),\dots\killing(\lambda_s))\in\Gamma_n(s)$.
\end{enumerate}
\end{proposition}

\begin{remark}\label{semi}
Let $\operatorname{Tens}_{n,\Bbb{Q}}(s)$ be the semigroup of $s$-tuples $(\lambda_1,\dots,\lambda_s)\subset (\frh_{n,\Bbb{Q}}^*)^s$  of dominant rational  weights of $\operatorname{SL}(n)$ such that  $(V_{N\lambda_1}\tensor V_{N\lambda_2}\tensor\dots \tensor V_{N\lambda_s})^{\SLL(n)}\neq 0$ for some positive integer  $N$. Then
$\operatorname{Tens}_{n,\Bbb{Q}}(s)$ is identified under the isomorphism \eqref{formule1} with $(\frh_{+,n,\Bbb{Q}})^s\cap \Gamma_n(s)$, we will denote the latter set by $\gammanq\subset \frh_{n,\Bbb{Q}}^s$. Therefore $\Gamma_n(s)$ is a rational polyhedral cone, and we need to determine extremal rays of this rational polyhedral cone.
\end{remark}

\subsection{The inequalities definining $\Gamma_n(s)$}
 To describe the Klyachko inequalities we introduce some notation:

\begin{defi}\label{closedsvv}
Let $I=\{i_1<\dots<i_r\}\subset [n]=\{1,\dots,n\}$ be a subset with $r$ elements, with $1\leq r\leq n-1$. This defines a Schubert variety $\Omega_I(F_{\bullet})$ in the Grassmannian $\operatorname{Gr}(r,n)=\operatorname{Gr}(r,\Bbb{C}^n)$
\begin{equation}\label{closedsv}
\Omega_I(F_{\bullet})=\{V\in \operatorname{Gr}(r,\Bbb{C}^n)\mid \dim V\cap F_{i_j}\geq j, \forall j\in[r]\}.
\end{equation}
 The cycle class of $\Omega_I(F_{\bullet})$, which lives in
 $H^{2|\sigma_I|}(\operatorname{Gr}(r,n)),$
is denoted by $\sigma_I$, here
\begin{equation}\label{codimI}
|\sigma_I|=\sum_{a=1}^r(n-r+a-i_a).
\end{equation}
Each $I$ as above also gives a permutation $w_I$ of $[n]$ as follows. Write $[n]-I=\{j_1,\dots,j_{n-r}\}$, Then $w_I(a)=i_a$ if $1\leq a\leq r$ and $w_I(a)=j_{a-r}$ if $r<a\leq n$.
\end{defi}

\begin{theorem}\label{KlBe1}\cite{Kly,Totaro,BLocal}
Suppose $x_1,\dots,x_s$ are $s$ elements in the Weyl chamber \eqref{weylC}. Then
$(x_1,x_2,\dots,x_s) \in \Gamma_n(s)$
if only if for  every tuple $(r,n,I_1,I_2,\dots,I_s)$ with $I_1$, $I_2$,\dots $,I_s$ subsets
of $[n]$ of cardinality $r$ each, with $1\leq r<n$ and
\begin{equation}\label{dagger}
\sigma_{I_1}\sigma_{I_2}\dots \sigma_{I_s}=[\operatorname{pt}]\in H^{2r(n-r)}(\Gr(r,n)),
\end{equation}
where $[\operatorname{pt}]$ is the cycle class of a point, the following inequality holds
\begin{equation}\label{evineq}
\sum_{j=1}^s\sum_{a\in I_j}x^{(a)}_j\leq 0.
\end{equation}
\end{theorem}

\subsection{Basic building blocks for extremal rays}\label{namon1}
The basic building blocks come about in the following way:
\begin{enumerate}
\item Fix $(r,n,I_1,I_2,\dots,I_s)$ satisfying  \eqref{dagger}. Then equality in inequality \eqref{evineq}, i.e.,
\begin{equation}\label{eveq}
\sum_{j=1}^s\sum_{a\in I_j}x^{(a)}=0.
\end{equation}
defines a facet $\mf$ of $\Gamma_n(s)$ by \cite{KTW}. Let $\mf_{\Bbb{Q}}=\mf\cap \gammanq$.
\item Pick a $j_0\in [s]$ and choose a set $T\subset [n]$ of cardinality $r$ so that $\Omega_{T}(F_{\bull})$ is a codimension one subvariety of $\Omega_{I_{j_0}}(F_{\bull})$ for any choice of flags $F_{\bull}$. This just means (see Lemma \ref{plum}) that $T=(I_{j_0}-\{a_0\})\cup\{a_0-1\}$ for some $a_0\in I_{j_0}$ such that $a_0>1$ and $a_0-1\not\in I_{j_0}$ (this index $a_0$ is determined by $I_{j_0}$ and $T$). We introduce the notation $I^{+,b}=(I-\{b\})\cup \{b-1\}$ if $b>1$, $b\in I$ and $b-1\not\in I$ (the plus in the exponent is to indicate that the codimension of the corresponding Schubert variety in $\Gr(r,n)$ has increased by one).
\item Let $A_k=I_k$ for $k\neq {j_0}, \ k=1,\dots,s$ and $A_{j_0}=T=I_{j_0}^{+,a_0}$.
\end{enumerate}
Therefore we have made a choice of the tuple $(r,n,I_1,\dots,I_s)$ satisfying \eqref{dagger} as well as the pair $(j_0,a_0)$.
\begin{defi}\label{basico}
Consider the locus
$$D=D(A_1,\dots,A_s)=D(I_1,\dots,I_{j_0-1},I_{j_0}^{+,a_0},I_{j_0+1},\dots,I_s)\subset \Fl(n)^s$$
consisting of $(F_{\bull}(1), F_{\bull}(2),\dots, F_{\bull}(s))$ such that
\begin{equation}\label{picky}
\bigcap_{i=1}^s \Omega_{A_i}(F_{\bull}(i))\neq \emptyset \subseteq \Gr(r,n).
\end{equation}
In Proposition \ref{DaDivisor}, we will show that $D$ is a $\op{SL}(r)$ invariant divisor in $\Fl(n)^s$. Note that $D(A_1,\dots,A_s)$ parameterizes ``special points'' of the moduli stack $\Fl(n)^s/\SLL(n)$: points  $(F_{\bull}(1), F_{\bull}(2),\dots, F_{\bull}(s))$ where the intersection \eqref{picky} is non-empty. We therefore refer to $D(A_1,\dots,A_s)$ as a ``modular intersection locus".
\end{defi}
Now $L=\mathcal{O}(D)$ is a line bundle on $\Fl(n)^s$ with a non-zero diagonal $\operatorname{SL}(n)$ invariant section, since $D$ is invariant under the diagonal action of $\SLL(n)$ on $\Fl(n)^s$. Write $L=L_{\lambda_1}\boxtimes L_{\lambda_2}\dots\boxtimes L_{\lambda_s}$.
Here $\lambda_i$ are dominant integral weights for $\SLL(n)$. This gives (by Proposition \ref{wellknown1} above), an element
\begin{equation}\label{pita}
[D(A_1,\dots,A_s)]=(\killing(\lambda_1),\killing(\lambda_2),\dots,\killing(\lambda_s))\in \gammanq.
\end{equation}
\begin{theorem}\label{one}
\begin{enumerate}
\item[(1)] The line bundle $L=\mathcal{O}(D)$ gives an extremal ray $\Bbb{Q}_{\geq 0} (\killing(\lambda_1),\killing(\lambda_2),\dots,\killing(\lambda_s))$ of $\gammanq$.
\item[(2)] This extremal ray lies on the facet $\mf_{\Bbb{Q}}$.
\item[(3)] The line bundle $L$ has the following rigidity property, reminiscent of a conjecture of Fulton\footnote{Let  $\lambda_1,\dots,\lambda_s$ be dominant integral weights of $\SLL(n)$, and set $f(N)=\dim (V_{N\lambda_1}\tensor V_{N\lambda_2}\tensor\dots\tensor V_{N\lambda_s})^{\SLL(n)}$. Fulton conjectured that $f(1)=1$ implies $f(N)=1$ for all positive integers $N$. This conjecture was proved in \cite{KTW}. A generalization to all groups appears in \cite{BKR}.}:
$\dim H^0(\Fl(n)^s,L^N)^{\SLL(n)}=1$ for all positive integers $N$. That is,
 $(V_{N\lambda_1}\tensor V_{N\lambda_2}\tensor\dots\tensor V_{N\lambda_s})^{\SLL(n)}$ is one-dimensional for all positive integers $N$.
\end{enumerate}
\end{theorem}

  The following proposition gives formulas for $\lambda_i=(\lambda_i^{(1)},\dots,\lambda_i^{(n)})$, $i=1,\dots,s$, where we set $\lambda_i^{(n)}=0$. For a subset $A\subset[n]$ of cardinality $r$, and $b\in A$ such that $b<n$ and $b+1\not\in A$, introduce
  the notation $A^{-,b}=(A-\{b\})\cup\{b+1\}$ (the minus sign is to indicate that the codimension decreases under the operation ).
\begin{proposition}\label{egreggy}
For any $b\in [n]$, $b<n$,
\begin{enumerate}
\item $\lambda^{(b)}_i-\lambda^{(b+1)}_i=0$ if $b\not\in A_i$.
\item  $\lambda^{(b)}_i- \lambda^{(b+1)}_i=0$ if $b\in A_i$ and $b+1\in A_i$.
\item Suppose $b\in A_i$ but $b+1\not\in A_i$.
$\lambda^{(b)}_i-\lambda^{(b+1)}_i=c_{i,b}$, where $c_{i,b}$ is the (possibly zero) intersection number
$$\sigma_{A_1}\dots\sigma_{A_{i-1}}\sigma_{A_i^{-,b}}\sigma_{A_{i+1}}\dots \sigma_{A_s}=c_{i,b}[\operatorname{pt}]\in H^{2r(n-r)}(\Gr(r,n)).$$
\end{enumerate}
\end{proposition}
\begin{remark}\label{egreggyremark}
 The formula for $\lambda_i$ in Proposition \ref{egreggy} can be written in terms of dominant fundamental weights $\omega_b$ (here $\omega_b$ is the the $b$th exterior power of the standard representation $\Bbb{C}^n$ of $\SLL(n)$): $\lambda_i = \sum_{1\leq b<n,b\in A_i, b+1\not\in A_i} c_{i,b}\omega_b$,  where $c_{i,b}$ is the (possibly zero) intersection number
$$\sigma_{A_1}\dots\sigma_{A_{i-1}}\sigma_{A_i^{-,b}}\sigma_{A_{i+1}}\dots \sigma_{A_s}=c_{i,b}[\operatorname{pt}]\in H^{2r(n-r)}(\Gr(r,n)).$$
\end{remark}

\begin{example}\label{thesis}
 Consider $r=2, n=4$ with  $I_1=\{2,3\}$, $T=\{1,3\}$, $I_2=I_3=\{2,4\}$, and ${j_0}=1$. It is easy to see that we get $\lambda_1= \omega_1+\omega3=(2,1,1,0)$ and $\lambda_2=\lambda_3=\omega_2=(1,1,0,0)$. The corresponding extremal ray of $\Gamma_4(3)$ is generated by $$(\kappa(\lambda_1),\kappa(\lambda_2),\kappa(\lambda_3))= ((1,0,0,-1),(\frac{1}{2},\frac{1}{2},-\frac{1}{2},-\frac{1}{2}), (\frac{1}{2},\frac{1}{2},-\frac{1}{2},-\frac{1}{2}).$$
Another example is given in Example \ref{another} (also see Section \ref{tonne}).
\end{example}
\subsection{Other extremal rays}\label{namon2}
To get all extremal rays, we first note that any extremal ray of $\gammanq$ lies on at least one  regular facet $\mfq$, given by equality in one of the inequalities \eqref{evineq} (see Lemma \ref{faceto}).

Now, fix $(r,n,I_1,I_2,\dots,I_s)$ satisfying \eqref{dagger}, and let $\mfq$ be the regular facet of $\gammanq$ defined by equality in the inequality \eqref{evineq}. We pose some problems refining Problem \ref{mainp}.
\begin{problem}\label{mainp2}
\begin{enumerate}
\item Find all extremal rays of the facet $\mfq$ of $\gammanq$. These will also be extremal rays of $\gammanq$.
\item Describe the entire facet $\mfq$ in terms of the (smaller) eigencones $\Gamma_{r,\Bbb{Q}}$ and $\Gamma_{n-r,\Bbb{Q}}$.
\end{enumerate}
\end{problem}

Let $q$ be the total number of our building block divisors $D(A_1,\dots,A_s)$ arising from $(r,n,I_1,I_2,\dots,I_s)$: That is the number of choices of pairs $(j_0,a_0)$, with $a_0\in I_{j_0}$ such that $a_0>1$ and $a_0-1\not\in I_{j_0}$.
(given this pair, as before $A_i=I_i$ for $i\neq j_0$ and $A_{j_0}=(I_{j_0}-\{a_0\})\cup \{a_0-1\}$.)
 These $q$ divisors will be shown to give linearly independent elements in $\Pic(\Fl(n)^s)$ (Lemma \ref{shew}).
\begin{defi}\label{dmfq2}
Define a  $\mf_2$ of $\mf$  as follows: $\mf_2\subseteq \mf$ is the set of $s$-tuples $(x_1,\dots,x_n)$, with
$x_i=(x_i^{(1)},\dots,x_i^{(n)})$ such that:
\begin{itemize}
\item For all $i\in [s]$ and $b\in I_i$ such that $b>1$ and $b-1\not\in I_i$, we have $x^{(b)}_i=x^{(b-1)}_i$.
\end{itemize}
Define $\mf_{2,\Bbb{Q}}=\mf_{\Bbb{Q}}\cap \mf_2$.
\end{defi}

\begin{theorem}\label{two}
The facet  $\mfq$ is naturally  a product $\mf\cong (\Bbb{Q}_{\geq 0})^q\times \mfqtwo.$ Therefore extremal rays of $\mfq$ are either the $q$ rays described above, or the extremal rays of $\mfqtwo$.
\end{theorem}
This reduces Problem \eqref{mainp2} to the problem of finding the extremal rays of $\mfqtwo$.

\subsection{Formulas for induction}\label{fmla1}

Note that $\frh_{r,\Bbb{Q}}\times \frh_{n-r,\Bbb{Q}}\subseteq \frh_{n,\Bbb{Q}}$ (withot any conditions on dominance), since $\SLL(r)\times \SLL(n-r)$ is a Levi subgroup of $\SLL(n)$. Explicitly given $y=(y^{(1)},\dots,y^{(r)})\in \frh_{r,\Bbb{Q}}$ and $z=(y^{(1)},\dots,y^{(n-r)})\in \frh_{n-r,\Bbb{Q}}$, we map $(y,z)$ to $$x=(y,z)=(y^{(1)},\dots,y^{(r)},z^{(1)},\dots,z^{(n-r)})\in \frh_{n,\Bbb{Q}}$$
This leads to map $\frh_{r,\Bbb{Q}}^s\times \frh_{n-r,\Bbb{Q}}^s\subseteq \frh_{n,\Bbb{Q}}^s$ by sending
$$((y_1,\dots,y_s), (z_1,\dots,z_s)), \mapsto (x_1,\dots,x_s),\ x_i=(y_i,z_i), i=1,\dots,s $$
Therefore $\Gamma_{r,\Bbb{Q}}(s)\times \Gamma_{n-r,\Bbb{Q}}(s)\subseteq \frh_{r,\Bbb{Q}}^s\times \frh_{n-r,\Bbb{Q}}^s$ sits naturally in $\frh_{n,\Bbb{Q}}^s$.

\begin{theorem}\label{dirac1}
For $(x_1,\dots,x_s)$ in $\Gamma_{r,\Bbb{Q}}(s)\times \Gamma_{n-r,\Bbb{Q}}(s)\subseteq \frh_{n,\Bbb{Q}}^s$.
For each $i$, consider ``naive'' induction $y_i\in \frh_{n,\Bbb{Q}}$ of $x_i$: the element
$y_i=(y_{i}^{(1)},\dots,y_i^{(n)})$, with
$$y_i^{(b)} =x_i^{(w_{I_i}^{-1}(b))},b=1,\dots,n.$$

Here $w_{I_i}$ are the elements of the Weyl group of $\op{SL}(n)$ regarded as permutation of $\{1,\dots,n\}$ given in Definition \ref{closedsvv}.

Define $\operatorname{Ind}(x_1,\dots,x_s)\in \frh_{n,\Bbb{Q}}^s$ as the sum (see equation \eqref{pita}),
\begin{equation}\label{sumq1}
(y_1,\dots,y_s) +\  \sum_{i=1}^s\sum_{b>1,b\in I_i, b-1\not\in I_i}(y_i^{(b)}-y^{(b-1)})[D(I_1,\dots,I_{i-1}, I_i^{+,b}, I_{i+1},\dots,I_s)]
\end{equation}
where we have used the notation $I_i^{+,b}=(I_i-\{b\})\cup\{b-1\}$.

Then, $\operatorname{Ind}$ defines a surjective (linear) induction map of cones which has a section
\begin{equation}\label{process1}
\operatorname{Ind}: \Gamma_{r,\Bbb{Q}}(s)\times \Gamma_{n-r,\Bbb{Q}}(s)\twoheadrightarrow\mfqtwo\  \subseteq\  \mfq\subseteq \gammanq\subseteq \frh_{n,\Bbb{Q}}^s.
\end{equation}
\end{theorem}
This construction comes from the geometry of partial flag varieties, and the formula for induction above is deduced from geometry. Therefore the extremal rays of $\mfqtwo$ are images (by induction) of extremal rays of $\Gamma_{r,\Bbb{Q}}(s)\times \Gamma_{n-r,\Bbb{Q}}(s)$ under the map \eqref{process1} described by the formula \ref{sumq1}. In practical terms, one would have to take the images of the (finitely many) extremal rays of  $\Gamma_{r,\Bbb{Q}}(s)\times \Gamma_{n-r,\Bbb{Q}}(s)$ under the explicit map \eqref{process1}, and extract a subset which generates the image in $\mf_2$. The fibers of \eqref{process1} can be understood in terms of the ramification divisor in the Schubert calculus problem defining the face $\mf$.
\subsection{Possible generalizations}
It is tempting to compare this picture to  Harish-Chandra's ``Philosophy of cusp forms" (see e.g., \cite[Chapter 49]{bumpy})(also see the connection of the Hermitian eigenvalue problem to problems over $p$-adics in  \cite{FBulletin}, and \cite{BZ}). The induction operation is an analogue of parabolic induction, and the boundary of the eigencone $\Gamma_n(s)$ has been described
in terms of Levi subgroups. The cuspidal part of the eigencone is the set of points on it which are not on regular faces. In a function theoretic sense, the corresponding sections of line bundles vanish on suitably defined cusps of $\Fl(n)^s$, see Lemma \ref{cusplemma}.

We have given an inductive description of extremal rays of each of the regular facets $\mf$. Any extremal ray of $\Gamma_n(s)$ is on one of these regular facets. Therefore we have accounted for all extremal rays of $\Gamma_n$.
\begin{enumerate}
\item An extremal ray of $\Gamma_n$ may be on several facets $\mf$, and we have over counted the extremal rays above. To address this we will have to work with (standard) non-maximal parabolics and consider regular faces of arbitrary codimension.
Perhaps the  parabolic associated to an extremal ray is not unique, but the corresponding Levi subgroup is unique.
\item The map $\operatorname{Ind}$ of \eqref{process1} may not take extremal rays to extremal rays. But every extremal ray of $\mf_2$ is obtained by induction from an extremal ray. This is a familiar problem wherein parabolic induction may not take irreducibles to irreducibles. Perhaps passing to non-maximal parabolics would solve this problem, and restricting the induction operation to cuspidal objects (which roughly means to induct rays not on regular faces) may solve this problem. Or perhaps there is a richer structure of suitable Hecke algebras controlling this picture.
\item See how these processes of induction and our basic extremal ray classes \eqref{pita} work out for small values of $n$ (for very small values see Section \ref{tonne}).
\item Bound the sizes of the coefficients of generators of extremal rays (i.e., e.g., the numbers $c_{i,b}$ in Proposition \ref{egreggy}).
\item Study the cuspidal part (perhaps with some integrality conditions) of the eigencone.
\end{enumerate}
It is most natural to address these points in the setting of general groups and arbitrary (standard) parabolics, and we hope to return to these themes (including the multiplicative eigenvalue problem) in future work \footnote{See e.g., \cite{BICM,kumar,RICM} for surveys, and \cite[Section 7]{BLocal}, and \cite{thaddy} for some work on the problem of vertices in the multiplicative problems.}. It is also an interesting problem to interpret the explicit formulas for our basic extremal rays  (Proposition \ref{egreggy}), and the formulas for induction \eqref{sumq1} in terms of the Honeycombs of Knutson and Tao \cite{KT}.
\subsection{Acknowledgements}
 I thank A. ~Gibney, A. ~Kazanova and M. ~ Schuster for useful discussions on the related problem of vertices of the multiplicative eigenpolyhedron, S.~ Kumar for useful discussions and for  comments on an initial version of this work,  and N.~ Ressayre for showing me (in 2014) his interesting  example of an extremal ray for $\SLL(9)$. 

\section{Basic divisors and the corresponding extremal rays}
Through out this section, we will use notation from Section \ref{namon1}. In particular the data $(r,n,I_1,\dots,I_s)$ is fixed as is $(j_0,0)$. Without loss of generality we assume that $j_0=1$. Let $D=D(A^1,\dots,A^s)$ with $A_1=(I_1-\{a_0\})\cup\{a_0\}$, and $A_i=I_i$ for $2\leq i\leq s$. In view of Lemma \ref{Fproperty}, statement (1) follows from (3) in Theorem \ref{one}

We describe conditions under which divisor loci in $\Fl(n)^s$ give extremal rays of $\gammanq$:
\begin{lemma}\label{Fproperty}
 Let $E$ be a reduced, irreducible $\SLL(n)$ invariant effective divisor on $\Fl(n)^s$ such that
 $H^0(\Fl(n)^s,\mathcal{O}(NE))^{\SLL(n)}$ is one dimensional  for all positive integers $N> 0$. Then
 $\mathcal{O}(E)$ gives an extremal ray of $\gammanq$. In addition, $\mathcal{O}(E)$ cannot be written as a tensor product of line bundles which have non-zero invariant sections.
  \end{lemma}
 \begin{proof}
 Suppose $\mathcal{O}(NE)$ is isomorphic to $L\tensor L'$ where $L=L_{\lambda}\boxtimes L_{\mu}\boxtimes L_{\nu}$ and
 $L'=L_{\lambda'}\boxtimes L_{\mu'}\boxtimes L_{\nu'}$ are such that $H^0(\Fl(n)^s,L)^{\SLL(n)}\neq 0$ and $H^0(\Fl(n)^s,L')^{\SLL(n)}\neq 0$.
 Pick non-zero sections $s\in H^0(\Fl(n)^s,L)^{\SLL(n)}$ and $s'\in H^0(\Fl(n)^s,L')^{\SLL(n)}$ let $F$ and $F'$ be their divisors of zeroes. Therefore
 there is a section  of $H^0(\Fl(n)^s,\mathcal{O}(NE))^{\SLL(n)}$ whose divisor of zeroes is $F+F'$, but the former has only one non-zero invariant section
 up to scale. Therefore $NE=F+F'$ as effective Weil divisors. This implies that $F$ and $F'$ both have support in $E$. We can conclude the proof since $E$ is multiplicity
 free by assumption.
  \end{proof}

\subsection{Divisor loci in $\Fl(n)^s$}\label{beginner}
\begin{defi} Let $I\subseteq [n]$, $|I|=n$ and $F_{\bull}\in\Fl(n)$.
The open Schubert cell (define $i_0=0, i_{r+1}=n$)
$$\Omega^0_I(F_{\bullet})=\{V\in \operatorname{Gr}(r,\Bbb{C}^n)\mid \dim V\cap F_{a}=j, i_{j}\leq a < i_{j+1}, j=0,\dots,r\},$$
lies in the smooth locus of the normal variety $\Omega_I(F_{\bullet})$.
\end{defi}
\begin{proposition}\label{DaDivisor}
$D=D(A^1,\dots,A^s)$ is non-empty, irreducible, and codimension one in $\Fl(n)^s$, and is invariant under the action of $\SLL(n)$.
\end{proposition}
Before beginning the proof we introduce some notation which will be used at other places:
\begin{notation}\label{listo}
Define the universal intersection of closed Schubert varieties (see e.g., \cite[Section 5]{BKR} for the scheme structures)
$Y\subseteq \operatorname{Gr}(r,n)\times \Fl(n)^s$ by
$$Y=\{(V,F_{\bull}(1),F_{\bull}(2),\dots,F_{\bull}(s)): V\in \bigcap_{i=1}^s \Omega_{I_i}(F_{\bull}(i))\}$$
and $\widetilde{D}\subset Y$, the universal intersection
$$\widetilde{D}=\{(V,F_{\bull}(1),F_{\bull}(2),\dots,F_{\bull}(s)): V\in \bigcap_{i=1}^s \Omega_{A_i}(F_{\bull}(i))\}$$
 Let $\pi:Y\to \Fl(n)^s$ be the projection. Note that $D=\pi(\widetilde{D})$. Define a subscheme "the universal intersection of Schubert cells" $Y^0\subseteq Y\subseteq \operatorname{Gr}(r,n)\times \Fl(n)^s$
\begin{equation}\label{Schubert}
 Y^0=\{(V,F_{\bull}(1),\dots,F_{\bull}(s)):   V\in \bigcap_{i=1}^s \Omega^0_{I_i}(F_{\bull}(i))\}
 \end{equation}
 Clearly $\SLL(n)$ acts on $Y^0$. Let $R^0\subset Y^0$ be the ramification divisor of $\pi:Y^0\to \Fl(n)^s$ (note that $Y^0$ is smooth).
\end{notation}
\begin{proof} (Of Proposition \ref{DaDivisor})
It is easy to see (by counting dimensions) that $\widetilde{D}$ is a divisor in $Y$. Both $Y$ and $\widetilde{D}$ are generically reduced, and irreducible, and $Y$ is smooth at the generic point of $\widetilde{D}$ (all these are easy to see and follow from \cite[Lemma 5.2]{BKR}), furthermore, the
 map $\pi:Y\to \Fl(n)^s$ is birational. Let $Y^{\op{sm}}$ be the smooth locus of $Y$, and $R\subset Y^{\op{sm}}$ the ramification locus of $Y^{sm}\to \Fl(n)^s$.

  We now show that  $\widetilde{D}\cap Y^{\op{sm}}$ is not in $R$: Let $(V,F_{\bull}(1),F_{\bull}(2),\dots,F_{\bull}(s))\in Y^0$ be such that
 $V\in \bigcap_{i=1}^s \Omega^0_{I_i}(F_{\bull}(i))$ is a transverse point of intersection. We also assume that
 $\bigcap_{i=1}^s \Omega^0_{I_i}(F_{\bull}(i))=\bigcap_{i=1}^s \Omega_{I_i}(F_{\bull}(i))$.

 Hence  $\dim V\cap F_{a_0-1}(1)=\dim V\cap F_{a_0}(1) -1$ and $V\cap F_{a_0-2}(1)=V\cap F_{a_0-1}(1)$. Let $S$ be a $a_0-1$ dimensional subspace of $\Bbb{C}^n$ such that
 $F_{a_0-2}(1)\subset S\subset F_{a_0}(1)$ such that $V\cap S=V\cap F_{a_0}(1)$. Define a new complete flag
 $F'_{\bull}$ such that $F'_{b}=F_{b}(1)$ for all $b\neq a_0$ and $F'_{a_0}=S$.

 It is now easy to see that
 $$V\in \Omega_{T}(F'_{\bull})\cap \bigcap_{i=1}^s\Omega^0_{I_i}(F_{\bull}(i))$$
  and is the only point in the  intersection
  $$\Omega_{I_1}(F'_{\bull})\cap \bigcap_{i=2}^s\Omega^0_{I_i}(F_{\bull}(i))=\bigcap_{i=1}^s \Omega^0_{I_i}(F_{\bull}(i)).$$ Therefore $(V,F'_{\bull},F_{\bull}(2),\dots,F_{\bull}(s))$  is in $\widetilde{D}\cap Y^{\op{sm}}$ and not in the ramification locus $R$.

 Since $\widetilde{D}$ is generically reduced and irreducible, and $\widetilde{D}\cap Y^{\op{sm}}$ is not contained in $R$, we can conclude the $D=\pi(\widetilde{D})$ (set theoretically) is codimension one, and irreducible, in $\Fl(n)^s$. In fact $D=\pi_*[\widetilde{D}]\in A^1(\Fl(n)^s)$ (the cycle theoretic pushforward). The invariance under $\SLL(n)$ is clear from the definitions.
 \end{proof}

 \subsection{A basic theorem about invariant functions on $Y^0-R^0$}
  \begin{theorem}\label{may11}
 $H^0(Y^0-R^0,\mathcal{O})^{\SLL(n)}=\Bbb{C}$.
 \end{theorem}

\subsection{A basic diagram}
We now recall the notational setup of \cite{belkaleIMRN}, recasting some definitions in the language of stacks. This will be used in the proofs of
 Theorems \ref{one} and \ref{may11}.

\begin{defi}\label{defiSS}
Let $\mathcal{S}=\mathcal{S}(r,n-r)$ be the smooth Artin stack parameterizing data
$(V, F'_{\bull}(1),\dots, F'_{\bull}(s))$,  $(Q, F^{''}_{\bull}(1),\dots,,F^{''}_{\bull}(s))$ and
an isomorphism $\det V\tensor \det Q=\Bbb{C}$. Here $V$ and $Q$ are arbitrary vector spaces of
ranks $r$ and $n-r$ respectively, and $F'_{\bull}(i)$ (respectively $F^{''}_{\bull}(i)$) are complete flags on $V$ (respectively $Q$) for $i=1,\dots,s$.

There is a determinantal divisor $R_{\mathcal{S}}$ on $\mathcal{S}$  of pairs $(V, F'_{\bull}(1),\dots, F'_{\bull}(s))$ and $(Q, F^{''}_{\bull}(1),\dots,,F^{''}_{\bull}(s))$  such that there is a non-zero map $\phi:V\to Q$ such that for all $i=1,\dots,s$, and $b\in I_i$ we have setting $m=I_i\cap\{1,\dots,b\}$, $\phi(F'_{m}(i))\subseteq F^{''}_{b-m}(i)$ \cite{belkaleIMRN}.
\end{defi}

Given $(V,F_{\bull}(1),F_{\bull}(2),\dots,F_{\bull}(s))\in Y^0$, we get $s$ complete flags on $V$ (from intersecting the flags with $V$) , as well as $s$ complete flags on $Q=\Bbb{C}^n/V$, by taking images under the natural map $\Bbb{C}^n\to Q$. This produces a map $p:Y^0/\SLL(n)\to \mathcal{S}$ ($Y^0$ was defined in Equation \eqref{Schubert}, and $Y^0/\SLL(n)$ is the stack quotient).  There is also a ``direct sum'' mapping  $i:\mathcal{S}\to Y^0/\SLL(n)$, where we choose an isomorphism $V\oplus Q\to \Bbb{C}^n$ consistent with determinants, and transfer flags
on $V$ and $Q$ to $\Bbb{C}^n$ so that $V$ sits in the desired Schubert cells.
(For example $F_{i}(1)= F'_{b}(1)\oplus F''_{i-b}(1)$ where $b=I_1\cap \{1,\dots,i\}$).

Therefore we have a diagram (with $i'=\pi\circ i$, and $i\circ p$ is the identity on $\ms$)
\begin{equation}\label{basicdiagram}
\xymatrix{
& Y^0/SL(n)\ar[dl]^{\pi}\ar[dr]_{p}\\
\Fl(n)^s/SL(n)  & &     \mathcal{S}\ar@/_/[ul]_i\ar[ll]^{i'}}
\end{equation}

It is also known from the description of the ramification divisor in \cite{belkaleIMRN}, that the ramification divisor $R^0$ in $Y^0$, is the pull back of a determinantal divisor
$R_{\mathcal{S}}$ in $\mathcal{S}$ (see Definition \ref{defiSS}).
\subsection{Proof of Theorem 2.5}\label{may11proof}
 We prove Theorem \ref{may11} in two steps:
 \subsubsection{Step (a)}
We claim that any invariant function on $Y^0-R^0$ restricts to a constant on $\mathcal{S}-R_{\mathcal{S}}$ when pulled back via $i$. We show this by showing
$h^0(\mathcal{S},\mathcal{O}(mR_{\mathcal{S}}))=1$ for all integers $m>0$. This will be shown to be a consequence of a conjecture of Fulton proved by Knutson, Tao and Woodward as follows:

There is a tautological natural map $\Fl(r)^s/\SLL(r)\times \Fl(n-r)^s/\SLL(n-r)\to\mathcal{S}$,
The divisor $R'\subset \Fl(r)^s/\SLL(r)\times \Fl(n-r)^s/\SLL(n-r)$ pulled back from $R_{\mathcal{S}}\subset \mathcal{S}$ was identified in \cite{belkaleIMRN} (also see \cite[Section 6.3]{BK2} for a more complete form of this computation) to be a divisor of the line bundle $L_r\boxtimes L_{n-r}$, where, using Definition \ref{seeB},
\begin{itemize}
\item$L_r=L_{\mu(I_1)}\tensor L_{\mu(I_2)}\tensor\dots\tensor L_{\mu(I_s)}$
\item $L_{n-r}=L_{\mu'(I_1)}\tensor L_{\mu'(I_2)}\tensor\dots\tensor L_{\mu'(I_s)}$.
\end{itemize}
\begin{defi}\label{seeB}
Let $I=\{i_1<\dots<i_r\}$ be a subset of $[n]$. Then $\mu(I)$ is the $r\times (n-r)$ Young diagram
$(\mu_1\geq \mu_2\geq\dots\geq \mu_r)$ with $\lambda_a=n-r+a-i_a$ for $a=1,\dots,r$.
and $\mu'(I)$ is the $r\times (n-r)$ Young diagram which is the transpose of $(n-r-\mu_r,\dots,n-r-\mu_1)$ with $\mu_1,\dots,\mu_n$ as above.
That is, $\lambda'(I)$ is dual of the transpose of $\lambda(I)$.
\end{defi}
Now the rank of $(V_{\mu(I_1)}\tensor V_{\mu(I_2)}\tensor\dots\tensor V_{\mu(I_s)})^{\SL(r)}$ equals the multiplicity of the class of a point in the product  $\sigma_{I_1}\dots\sigma_{I_s}$, which is one by assumption. The rank of $(V_{\mu'(I_1)}\tensor V_{\mu'(I_2)}\tensor\dots\tensor V_{\mu'(I_s)})^{\SL(n-r)}$ is the multiplicity of the corresponding dual Schubert problem in
$\Gr(n-r,n)$, also one by Grassmann duality. These ranks are one even upon scaling all $\mu(I_i)$ etc by $m>0$  by Fulton's conjecture \cite{KTW} (also see Theorem \ref{shosty} (2)). Therefore the space of $\SLL(r)\times \SLL(n-r)$ invariant sections of $\mathcal{O}(mR')$ on $\Fl(r)^s\times \Fl(n-r)^s$ is one dimensional for any $m$. This completes the proof of (a).

\subsubsection{Step (b)} The value of any $\SLL(n)$ invariant regular function at a point $x\in Y^0-R^0$ coincides with its value at $i(p(x))$. This is true because a representative of $i(p(x))$ is in the orbit closure of $x$: If $x= (V,F_{\bull}(1),F_{\bull}(2),\dots,F_{\bull}(s))\in Y^0$, choose a direct sum decomposition $\Bbb{C}^n=V\oplus Q$, and let $\phi_t\in \SLL(n)$ be an automorphism which is multiplication by $t^{n-r}$ on $V$ and multiplication by $t^{-r}$ on $Q$. The point $i(p(x))=\lim_{t\to 0}\phi_t(x)$.

This completes the proof of Theorem \ref{may11}.

\subsection{Proof of Theorem 1.9}\label{lastpara}
It is easy to see from the birationality of $\pi:Y\to \Fl(n)^s$ that $Y^0-R^0\to \Fl(n)^s$ is an open embedding (see e.g., \cite[Section 3]{BKR}).
 We claim,
 \begin{itemize}
 \item  $\pi(Y^0-R^0)\subseteq \Fl(n)^s-D$.
 \end{itemize}
 If $x\in Y^0-R^0$ maps to a point $(F_{\bull}(1),\dots,F_{\bull}(s))\in D\subseteq \Fl(n)^s$, then the intersection of closed Schubert varieties $\cap_{i=1}^s\Omega_{I_i}(F_{\bull}(i))$ is disconnected with one isolated point given by the first coordinate of $x$, which is ruled out by Zariski's main theorem. This proves the claim.

 Therefore
 a  $\SLL(n)$ invariant function on $\Fl(n)^s-D$ restricts to an $\SLL(n)$ invariant function on $Y^0-R^0$,
 and is hence a constant by Theorem \ref{may11}. This completes the proof of Theorem \ref{one}(3).

Lemma \ref{Fproperty}, and Proposition \ref{DaDivisor}, show that $\Bbb{Q}_{\geq 0}(\killing(\lambda_1),\killing(\lambda_2),\dots,\killing(\lambda_s))$ is an extremal ray of $\gammanq$. This shows (1) of Theorem \ref{one}.

Clearly $\mathcal{O}(D)$ has an invariant  section $1$  that does not vanish on the image of $\mathcal{S}-R_{\mathcal{S}}\to \Fl(n)^s$ (which lands inside $\pi(Y^0-R^0)$, and hence avoids $D$). Therefore by  GIT, see e.g., \cite[Proposition 10]{BK}, the ray  $\Bbb{Q}_{\geq 0}(\killing(\lambda_1),\dots,\killing(\lambda_s))$ lies in $\mf$.
(The Mumford inequality \eqref{evineq} measures the order of vanishing of an invariant section at  $i(p(x))=\lim_{t\to 0}\phi_t(x)$.)
The technique used above, of showing that the line bundle $\mathcal{O}(D)$ gives a point of $\mf$ because it has an invariant  section that does not vanish on $\mathcal{S}-R_{\mathcal{S}}$ first appeared in \cite[Section 4]{R1}.

\begin{remark}\label{markJ}
The relations between   Mumford indices in GIT and Hermitian eigenvalue inequalities used above is shown  in \cite{Totaro} and \cite[Section 7.3]{BK}.
\end{remark}

\section{Cycle classes of the basic divisors}
\subsection{Universal cycle classes of Schubert varieties}
\begin{defi}\label{defalpha}
Let $1\leq b\leq n-1$. Define  $\alpha_b\in A^1(\Fl(n)), b=1,\dots,n-1$ as the first Chern class of the line bundle on   $\Fl(n)$ whose fiber over
$F_{\bullet}: 0\subsetneq F_1\subset F_2\subsetneq \dots \subsetneq F_{n}=\Bbb{C}^n$ coincides with $(\wedge^a F_b)^*$.
\end{defi}
\begin{defi} For every subset $A\subseteq [n]$ of cardinality $r$, define a universal
 Schubert variety $\Omega^{\operatorname{Univ}}_A\subset \Gr(r,n)\times \Fl(n)$,
 $$\Omega^{\operatorname{Univ}}_A=\{(V,F_{\bull})\mid V\in\Omega_I(F_{\bull})\}$$
\end{defi}
The codimension $m$ Chow group
$A^m(\Gr(r,n)\times \Fl(n))$  is a direct sum
\begin{equation}\label{directsum}
\oplus_{s=0}^m A^{m-s}(\Gr(r,n))\tensor A^s(\Fl(n))
\end{equation}
Here cycles on $\Gr(r,n))$ and $\Fl(n)$ are pulled back under the flat projections and multiplied in the Chow ring (which coincides with cohomology for
$\Gr(r,n)$ and $\Fl(n)$, and their products).

\begin{proposition}\label{cycleclass}
Let $A$ be a subset of $[n]$ of cardinality $r$.  Let $m=|\sigma_A|$.
\begin{enumerate}
\item The projection of the cycle class $[\Omega^{\operatorname{Univ}}_A]\in A^m(\Gr(r,n)\times \Fl(n))$ to the summand $s=0$
of the direct sum \eqref{directsum} is $\sigma_A$.
\item The projection of the cycle class $[\Omega^{\operatorname{Univ}}_A]\in A^m(\Gr(r,n)\times \Fl(n))$ to the summand (of the direct sum \eqref{directsum}) $A^{m-1}(\Gr(r,n))\tensor A^1(\Fl(n))$
is equal to
$$\sum_{b\in[n-1], b\in A,\ b+1\not\in A} \sigma_{A_b}\tensor  \alpha_b$$
where $A_b=(I-\{b\})\cup \{b+1\}$ and $\alpha_b$ is defined in Definition \ref{defalpha}.
\end{enumerate}
\end{proposition}

\begin{proof}
  To prove (1) intersect with  $\Omega_J(T_{\bull})\times p$ where $p=F_{\bull}$ is a fixed point of $\Fl(n)$ and $T_{\bull}$ a flag in general position with respect to $p$ and $|\sigma_J|+|\sigma_A|=r(n-r)$. The intersection number is zero unless $\sigma_A$ is dual to $\sigma_J$ under the intersection pairing. Therefore the cycle class is as stated.

To prove (2), let $a\in[n-1]$, and we intersect  $\Omega^{\operatorname{Univ}}_A$ with $\Omega_J(T_{\bull})\times C_b$ where $C_b$ is a rational curve of all complete flags $F_{\bull}$ with $F_j$ fixed for $j\neq b$ (so $F_b$ varies but is constrained to lie inside a fixed $F_{b+1}$ and containing a fixed $F_{b-1}$,  and $|\sigma_J|+|\sigma_A| =r(n-r)+1$.
The closed Schubert variety $\Omega_A(F_{\bull})$ does not depend upon the choice of $F_b$ if $b\not\in A$ or if $b\in A$ and $b+1\in A$. The intersection number is zero then for codimension reasons.

If $b\in A$ and $b+1\not\in A$, then the union of the subvarieties $\Omega_A(F_{\bull})$ as $F_{\bull}$ varies over the curve $C_b$ equals
$\Omega_{\widetilde{A}}(S_{\bull})$ where $S_{\bull}$ is a fixed point of $C_b$
and $\widetilde{A}=(A-\{b\})\cup\{b+1\}$. It is easy to see that the desired intersection
number counts number of points in $\Omega_{\widetilde{A}}(S_{\bull})\cap \Omega_J(T_{\bull})$, since for any point $V$ in this intersection
$\dim V\cap S_{b-1}=1+ \dim  V\cap S_{b+1} $, and therefore there is a
unique choice of $S_{b-1}\subset F_b\subset S_{b+1}$ such that
$V\cap F_b=V\cap S_{b+1}$.  Here we have used the fact that the flag
$T_{\bull}$ is general position with the fixed $S_{\bull}$. This yields the formula for the cycle class of the projection.
\end{proof}

\subsection{Proof of Proposition 1.10}\label{may12}
Let  $X=\Gr(r,n)\times \Fl(n)^s$.
Now let $\pi_1,\pi_2,\dots,\pi_s$ be the $s$ projections $X\to \Gr(r,n)\times \Fl(n)$. It is easy to see that $\yiy$ is the generically transverse intersection of $s$ subschemes of $X$:
\begin{equation}\label{above}
\yiy=\bigcap_{i=1}^s\pi_1^{-1}(\Omega^{\operatorname{Univ}}_{A_i})
\end{equation}
(Note $\yiy$ is generically smooth, and coincides with the scheme theoretic intersection above by definition, see e.g., \cite[Section 5]{BKR}.  For generic transversality,  we only have to verify the numerical codimension condition which
is easy.)

The class $[(D(A_1,\dots,A_s)]$ is $\pi_*[\yiy]$ where $\pi:X\to  \Fl(n)^s$. To get a non-zero contribution to $\pi_*[\yiy]$, we need the projection of  $[\Omega^{\operatorname{Univ}}_{A_i}]$
to the summands $s=0$ and $s=1$ in the direct sum decomposition of $A^{|\sigma_{A_i}|}(\Gr(r,n)\times \Fl(n)$ given by Proposition\ref{cycleclass}  for $i=1,\dots,s$. This is because the  degree of cycles on $\Gr(r,n)$ have to sum to $r(n-r)$ to give a non-zero push forward via $\pi$. We can therefore see that we need to pick the $s=0$ terms for all but one of the  $A_i$ and
one $s=1$ term.  The computations of Proposition \ref{cycleclass} now imply the desired formulas for $\lambda_1,\dots,\lambda_s$.

\begin{example}\label{another}
 The following is another example of the formulas in Proposition \ref{egreggy}. Let $r=5$, $n=8$, $s=3$ and $I_1=\{3,4,5,7,8\}\subset[8]$, $I_2=I_3=\{2,3,5,6,8\}\subset [8]$.  Let $(j_0,a_0)=(1,3)$ In can be checked using the Littlewood-Richardson rule that the relevant intersection number is one.  We get $\lambda_1=(3,3,2,2,2,0,0,0)=2\omega_5+\omega_2$
and $\lambda_2=\lambda_3=(4,4,4,2,2,2,0,0)=2\omega_6+ 2\omega_3$.
\end{example}
\begin{remark}
The proof of Proposition \ref{egreggy} shows that it is valid in a more general context: Let $A_1,\dots,A_s$ be subsets of $[n]$ each of cardinality $r$ such that $\sum_{j=1}^s |\sigma_{A_i}|=r(n-r)+1$. Then we can define a divisor class in $A^1(\Fl(n)^s)$ as the push forward (which is possibly zero) of the fundamental class $[\yiy]$ where $\yiy$ is defined as in
\eqref{above}. This class is supported on the image of $\yiy\to \Fl(n)^s$. The formulas in Proposition \ref{egreggy}
give this divisor class in this generality.
\end{remark}

\section{The facet $\mf$ as a product}

\subsection{Some elementary observations}
\begin{lemma}\label{faceto}
\begin{enumerate}
\item Any extremal ray of $\gammanq$ lies on a regular facet of $\gammanq$ (i.e., the facet is not obtained as intersection of $\gammanq$ with a chamber wall).
\item Any extremal ray of $\gammanq$ lies on a Weyl chamber wall of $(\frh_{+})^s$.
\end{enumerate}
\end{lemma}
\begin{proof}
Let $R=\Bbb{Q}_{\geq 0}(x_1,\dots,x_n)$ be an extremal ray in $\gammanq$ which is not on any of the reqular facets of $\gammanq$. It is then easy to see that $R$ is then an extremal ray of $(\frh^{+}_{n,\Bbb{Q}})^s$, which is necessarily of the form $\Bbb{Q}_{\geq 0}(x_1,\dots,x_s)$ with exactly one of the $x_i$ non-zero. This implies that exactly one of the $s$ matrices $A_1,\dots,A_s$ in Definition \ref{nafnew} is non-zero. But this leads to contradiction since $\sum A_j=0$. This proves the first part.

For the second part (which is well known) we proceed as follows. Suppose the contrary. It  lies on a regular facet $\mf$ in (1). We look at the corresponding Hermitian eigenvalue problem: The matrices $A_1,\dots,A_s$ with $\sum A_i=0$ with eigenvalues $x_1,\dots,x_s$ can then be assumed to  preserve the corresponding  direct sum decomposition $\Bbb{C}^n=\Bbb{C}^r\oplus\Bbb{C}^{n-r}$. We obtain solutions of the eigenvalue problems $(A'_1,\dots,A'_s)$ (these are $r\times r$ matrices) and $(A_1'',\dots,A_s'')$ respectively (these are $(n-r)\times (n-r)$ matrices) with $\sum A'_i=0$ and $\sum A_i''=0$ (without trace free conditions on these matrices).

Consider the one parameter family $A'_1(t)=A'_1+\frac{t}{r}I_r$, $A'_2(t)=A'_2-\frac{t}{r}I_r$, $A'_i(t)=A'_i$ for $i>2$; and $A_1''(t)=A_1''-\frac{t}{n-r}I_{n-r}$, $A_2''(t)=A_2''+\frac{t}{n-r}I_{n-r}$ and $A_i''(t)=A_i''$, $i>2$ for $-\epsilon<t<\epsilon$ for small $\epsilon >0$ and $t$ rational.  Let $A_i(t)$ be the direct sum $s$-tuple of $n\times n$ matrices, a tuple of traceless Hermitian matrices which sum to zero. The original ray corresponded to $t=0$, and we can deform it linearly inside the eigencone $\Gamma_n(s)$ in two opposite directions, a contradiction. The assumption of regularity of $R$ implies that we know the ordered eigenvalues of $A_i(t)$ for $|t|<\epsilon$ with small $\epsilon >0$ (not that $A_1(t)\neq A$, or a scalar multiple,  since some of the eigenvalues of $A_1(t)$ have increased, while others have decreased).
\end{proof}

\subsection{Two types of rays in $\mathcal{F}$}\label{feast}

As in the Introduction fix $(r,n,I_1,\dots,I_s)$ satisfying \eqref{dagger}. This gives rise to a facet $\mf$ of $\Gamma_n(s)$, and a facet $\mf_{\Bbb{Q}}$ of $\gammanq$ given by equality in inequality \eqref{evineq}.

We distinguish two types of rays of $\mathcal{F}$:
A ray  $\Bbb{Q}_{\geq 0} (x_1,\dots,x_n)\in \mathcal{F}_{\Bbb{Q}}$ is called a type I ray of $\mf_{\Bbb{Q}}$ if there
exists $j\in [s]$ and a $b>1$ such that $b\in I_j$, $b-1\not\in I_j$, and  and $x^{(b)}_{j}\neq x^{(b-1)}_j$.

A ray $\Bbb{Q}_{\geq 0}(x_1,\dots,x_n)\in \mathcal{F}_{\Bbb{Q}}$ is called a type II ray of $\mf_{\Bbb{Q}}$  if it is not a type I ray of $\mf_{\Bbb{Q}}$. Points on Type II rays of $\mf_{\Bbb{Q}}$
form a polyhedral subcone $\mf_{2,\Bbb{Q}}$, which is a face of $\mathcal{F}_{\Bbb{Q}}$ (as in Definition \ref{dmfq2}).

An extremal ray of $\gammanq$ may lie on two different facets. It is possible that it is a type I ray of one of these facets, and a type II ray of the other (see Section \ref{tonne} for an example).
\subsubsection{Proof of Theorem 1.15}
Let us first note that the basic extremal rays obtained from $(r,n,I_1,\dots,I_s)$ and a choice of $j_0$ and $a_0$ are type I on the face $\mf$ defined by  $(r,n,I_1,\dots,I_s)$. This can be seen by inspecting Proposition \ref{egreggy} and noticing that $\lambda_{j_0}^{(a_0-1)}-\lambda_{j_0}^{(a_0)}=1$.

\begin{lemma}\label{shew}
The $q$ basic extremal rays on $\mf$ are linear independent.
\end{lemma}
\begin{proof}
This follows from the above jumping by one at $a_0$ phenomenon, and the following observation: Consider a basic extremal ray coming from the data $(r,n,I_1,\dots,I_s)$ and $j_0,a_0$. Suppose $(j'_0,a'_0)$ is a different pair producing a basic extremal
ray. Then writing
$D=D(I_1,\dots,I_{j-1},(I_j-\{a_0\})\cup\{a_0-1\},I_{j+1},\dots,I_s)$, and $\mathcal{O}(D)=L_{\lambda_1}\boxtimes L_{\lambda_2}\dots\boxtimes L_{\lambda_s}$, we have
$\lambda_{j'_0}^{(a'_0-1)}-\lambda_{j'_0}^{(a'_0)}=0$, since $a'_0-1\not\in A_{j'_{0}}$ in the formula for $\lambda_{j'_0}$ in Proposition \ref{egreggy}.
\end{proof}

Let $\delta_1,\dots,\delta_q$ be the images in $\mf_{\Bbb{Q}}$ be the images (see \eqref{pita}) of the basic extremal rays.
It is also easy now to see that the sum mapping
$$(\Bbb{Q}_{\geq 0})^q\times\mf_{2,\Bbb{Q}}\to \mf_{\Bbb{Q}}, (a_1,\dots,a_q, x)\mapsto (\sum_{i=1}^q a_i \delta_i) + x, a_i\geq 0, x\in \mf_{2,\Bbb{Q}}\to \mf_{\Bbb{Q}}$$
is injective. To show the surjection (and hence complete the proof of Theorem \ref{two}) we prove the following more refined statement:
\begin{proposition}
Suppose $\mu_1,\dots,\mu_s$ are dominant integral weights for $\SLL(n)$, $j_0\in[s]$ and $a_0\in I_{j_0}$ such that $a_0-1>0$ and $a_0-1\not\in I_{j_0}$. Let
$$D=D(I_1,\dots,I_{j_0-1}, (I_{j_0}-\{a_0\})\cup \{a_0-1\}, I_{j_0+1},\dots,I_s)$$
 be our basic divisor \eqref{basico}.
Assume further that
\begin{enumerate}
\item[(1)] $(\killing(\mu_1),\dots,\killing(\mu_s))\in \mf_{\Bbb{Q}}$
\item[(2)] $\mu_{j_0}^{(a_0)}-\mu_{j_0}^{(a_0-1)}\neq 0$ (therefore $(\killing(\mu_1),\dots,\killing(\mu_s))\not\in\mf_2$).
\end{enumerate}
Let $N=L_{\mu_1}\tensor \dots\tensor L_{\mu_s}$. Then
$$H^0(\Fl(n)^s,N)^{\SLL(n)}= H^0(\Fl(n)^s,N(-D))^{\SLL(n)}$$
\end{proposition}
\begin{proof}
Let $s\in H^0(\Fl(n)^s,N)^{\SLL(n)}$. We need to show that $s$ vanishes on any point $(F_{\bull}(1)\dots,F_{\bull}(s))\in D$. Pick $V$ is in intersection \eqref{picky}.

Let $\killing(\mu_i)=(y_i^{(1)},\dots,y_i^{(n)})$, $i=1,\dots,s$. The semistability inequality corresponding to $V$ necessarily fails (see \eqref{ssinequality}, and Remark \ref{remy} below), since (using Assumption (2) above)
$$\sum_{b\in (I_{j_0}-\{a_0\})\cup \{a_{0}-1\}}y_{j_0}^{(b)} > \sum_{b\in I_{j_0}} y_{j_0}^{(b)}$$
Invariant sections vanish at non-semistable points (this is the definition of semistability), and the desired statement follows.
\end{proof}
\subsection{Conclusion of proof of Theorem 1.15}
If $(y_1,\dots,y_s)\in\mfq-\mfqtwo$, after scaling we can assume that we have dominant integral weights $\mu_1,\dots,\mu_s$ for $\SLL(n)$ such that $(\killing(\mu_1),\dots,\killing(\mu_s))=(y_1,\dots,y_s)$. We can also assume that defining $N=L_{\mu_1}\tensor \dots\tensor L_{\mu_s}$,  $H^0(\Fl(n)^s,N)^{\SLL(n)}\neq 0$. Pick an $(j_0,a_0)$ such that
$\mu_{j_0}^{(a_0)}-\mu_{j_0}^{(a_0-1)}\neq 0$ (and $a_0\in I_{j_0}$, $a_0>1$, and $a_0-1\not\in I_{j_0}$).  Applying the above proposition we can write $(y_1,\dots,y_s)$ as the sum of two points in $\mfq$ (one a basic extremal ray), and repeat this procedure to conclude the proof of Theorem \ref{two}.
\begin{remark}\label{remy}
A point of $(F_{\bull}(1),\dots,,F_{\bull}(s))\in \Fl(n)^s$ gives a filtered vector space (see \cite{FBulletin}) structure on $\Bbb{C}^n$. We call this filtered vector space semistable
for the weights $(x_1,\dots,x_n)$ if every subspace $V\subset \Bbb{C}^n$ has the following property: If we determine subsets $I_1, I_2, \dots, I_s$ of $[n]$ each of cardinality
$r=\dim V$ such that
$V\in \bigcap_{i=1}^s\Omega^0_{I_i}(F_{\bull}(i))$ then
\begin{equation}\label{ssinequality}
\sum_{i=1}^s\sum_{b\in I_i} x_i^{(b)}\leq 0.
\end{equation}
Geometric invariant theory (see e.g., \cite{FBulletin}) shows that  if $x_i=\killing(\lambda_i)$, as above then $$x=(F_{\bull}(1),\dots,F_{\bull}(s))\in \Fl(n)^s$$ is semistable for the weights $x_1,\dots,x_n$
if and only if there is a $m$ and a section $s\in H^0(\Fl(n)^s,L)^{\SLL(n)}$ which does not vanish at $x$, with
$L=L_{\lambda_1}\tensor \dots\tensor L_{\lambda_s}$
\end{remark}
\subsection{Cusps and Vanishing}
The following is an immediate consequence of the meaning of semistability (see e.g., Proposition 10 in \cite{BK}):
\begin{lemma}
Let $(\lambda_1,\dots,\lambda_s)$ be a $s$-tuple of dominant integral weights such that setting $x_i=\killing(\lambda_i)$ for $i=1,\dots,s$,
\begin{enumerate}
\item[(1)] $(x_1,\dots,x_s)\in \Gamma_n(s)$.
\item $(x_1,\dots,x_s)$ is not on the facet $\mf$ of  $\Gamma_n(s)$ given by $(r,n,I_1,\dots,I_s)$.
\end{enumerate}
Then any section of $H^0(\Fl(n)^s,L_{\lambda_1}\tensor\dots \tensor L_{\lambda_s})^{\SLL(n)}$ vanishes on the image of the map $i'$ in the diagram \ref{basicdiagram}.
\end{lemma}
We may call the images of the map $i'$ in \eqref{basicdiagram} over all choices of $(r,n,I_1,\dots,I_s)$ satisfying \eqref{dagger}, the cusps of $\Fl(n)^s$. Therefore,

 \begin{lemma}\label{cusplemma}
 Suppose  $(\lambda_1,\dots,\lambda_s)$ is such that
$(\killing(\lambda_1),\dots,\killing(\lambda_s))$ is in  $\Gamma_n(s)$ but not on any regular facet. Then,
 then any invariant  section in $H^0(\Fl(n)^s,L_{\lambda_1}\tensor\dots \tensor L_{\lambda_s})^{\SLL(n)}$ vanishes
 at all cusps of $\Fl(n)^s$.
\end{lemma}

\section{Induction operations}\label{rumble}
Fix $(r,n,I_1,\dots,I_n)$ satisfying satisfying \eqref{dagger}. This gives rise to a facet $\mf_{\Bbb{Q}}$ of $\gammanq$ given by equality in inequality \eqref{evineq}.
\begin{theorem}\label{egregium}
We have a surjection of cones with a section (as in \eqref{process1})
\begin{equation}\label{sibelius}
\operatorname{Ind}: \Gamma_{r,\Bbb{Q}}\times \Gamma_{n-r,\Bbb{Q}}\twoheadrightarrow \mf_{2,\Bbb{Q}}\ \subseteq \mf_{\Bbb{Q}}
\end{equation}
obtained as the composition of the following maps (the maps and  terms that appear here are defined below in Section \ref{bellow})
\begin{equation}\label{sibelius2}
\Gamma_{r,\Bbb{Q}}\times \Gamma_{n-r,\Bbb{Q}}\leto{\sim}\Pic^{+}_{\Bbb{Q}}(\mathcal{\ms})\twoheadrightarrow\Pic^{+}_{\Bbb{Q}}(\mathcal{\ms}-R_{\ms}) \letof{\sim,\operatorname{Ind}}\Pic^{+,\deg=0}_{\Bbb{Q}}(\mathcal{B})\leto{\sim}\mfqtwo
\end{equation}
An explicit formula for the composite is given in Theorem \ref{dirac1}.
 \end{theorem}

\subsection{A description of terms that appear in equation (5.2), and a overall sketch of proof}\label{bellow}
The map \eqref{sibelius} comes about by putting together several identifications and a geometric induction operation (using Definitions \ref{biggy} and \ref{longy}):
\begin{enumerate}
\item $\Gamma_{r,\Bbb{Q}}\times \Gamma_{n-r,\Bbb{Q}}$ is identified with $\Pic^{+}_{\Bbb{Q}}(\mathcal{\ms})$, the group of line bundles on $\ms$ (tensored with $\Bbb{Q}$) such that
some power has a non-zero global section, see Definitions \ref{defiSS} and \ref{biggy} and Proposition \ref{stbe}.
\item We consider some partial flag varieties $\Fl(I_1),\dots, \Fl(I_s)$, and define a stack $\mathcal{B}=(\Fl(I_1)\times\Fl(I_2)\dots\times \Fl(I_s))/\SL(n)$ in Definition \ref{bee}.
We show that  $\mfqtwo$ can be identified with the intersection of
$\Pic^+_{\Bbb{Q}}(\mathcal{B})$ with a hyperplane  $\Pic^{\deg =0}_{\Bbb{Q}}(\mathcal{B})\subseteq\Pic_{\Bbb{Q}}(\mathcal{B})$. This hyperplane is given by $(\lambda_1, \dots,,\lambda_s)$ such that $(\killing(\lambda_1), \dots \killing(\lambda_s))$  satisfies equality in  inequality \eqref{evineq}.  We denote this
intersection by $\Pic^{+,\deg=0}_{\Bbb{Q}}(\mathcal{B})$ which is therefore identified with $\mfqtwo$ (see Proposition \ref{secondthing}).  The relevance of partial flag varieties to type II rays of $\mfq$  is pointed out in Remark \ref{forcing}.
\item We construct a basic geometric induction operation
$$\operatorname{Ind}: \operatorname{Pic}(\mathcal{S}-R_{\mathcal{S}})\to \Pic(\mathcal{B})$$
and show that it  gives an isomorphism between $\Bbb{Q}$-vector spaces
\begin{equation}\label{morning}
\Pic^{\deg=0}_{\Bbb{Q}}(\ms-R_{\ms})\leto{\sim}\Pic^{\deg =0}_{\Bbb{Q}}(\mathcal{B})
\end{equation}
where the former is the group of line bundles on which, multiplication by scalars ($t^{n-r}$ on $V$ and $t^{-r}$ on $Q$, $t\in \Bbb{C}$) acts trivially. The isomorphism \eqref{morning}  will be shown (Theorem \ref{shosty} (3)) to induce an linear cone bijection:
$$\Pic^{+}_{\Bbb{Q}}(\mathcal{\ms}-R_{\ms})\leto{\sim}\Pic^{+,\deg=0}_{\Bbb{Q}}(\mathcal{B})$$
\item Finally, we show $\Pic^{+}_{\Bbb{Q}}(\mathcal{\ms})$ surjects onto $\Pic^{+}_{\Bbb{Q}}(\mathcal{\ms}-R_{\ms})$ with a section (Proposition \ref{bloodtest}).
\end{enumerate}

\section{Picard groups}

\begin{defi}\label{biggy}
Let $\mathcal{X}$ be an Artin stack (e.g., $\ms$ or $\ms-R_{\ms}$).
Let $\Pic^{+}(\mathcal{X})$ be the semigroup of all line bundles on $\mathcal{X}$ which have non-zero global sections,  $\Pic_{\Bbb{Q}}(\mathcal{X})= \Pic(\mathcal{X})\tensor \Bbb{Q}$, and $\Pic^{+}_{\Bbb{Q}}(\mathcal{X})\subset  \Pic_{\Bbb{Q}}(\mathcal{X})$ be the rational effective cone of all $\Bbb{Q}$ rational line bundles on $\mathcal{X}$ such that some tensor power has a  non-zero section.

\end{defi}

\begin{defi}\label{longy}
$\Bbb{C}^*$ acts on every  point of $\mathcal{S}$ as follows: $t\in\Bbb{C}^*$ acts on $V$ by multiplication by $t^{n-r}$ and on $Q$ by
$t^{-r}$.  Therefore $\Bbb{C}^*$ acts on the fibers of any line bundle
on $\ms$ (or on $\ms-R_{\ms}$).  Let $\Pic^{\deg =0}(\ms)$
and $\Pic^{deg=0}(\ms-R_{\ms})$ denote the group of line bundles where this action of $\Bbb{C}^*$ is trivial. It is clear that $\Pic^{deg=0}(\ms-R_{\ms})$ contains $\Pic^{+}(\ms-R_{\ms})$.
\end{defi}

\begin{defi}\label{line}
Given $\lambda=(\lambda^{(1)},\dots,\lambda^{(n)})\in\Bbb{Z}^n$, we get a $\GL(n)$ equivariant line bundle $L_{\lambda}$ on $\Fl(n)$ whose fiber at a point $F_{\bull}$ (with $\lambda^{(n+1)}=0$)
is $$\bigotimes_{a=1}^{n} (\det(F_a)^{\lambda^{(a)}-\lambda^{(a+1)}})^*=\bigotimes_{a=1}^{n} ((F_a/F_{a-1})^{\lambda^{(a)}})^*$$
The space of sections $H^0(\Fl(n),L_{\lambda})$ equals $V^*_{\lambda}$ as a representation of $\operatorname{GL}(n)$ if
$\lambda$ is dominant (i.e., $\lambda^{(i)}$ are weakly decreasing), and zero otherwise.
\end{defi}
\subsection{Picard group of $\Fl(n)^s/\SLL(n)$}\label{repeat}
Let $\mathcal{A}=\Fl(n)^s/\SLL(n)$. The Picard group of $\mathcal{A}$ is the set of line bundles on $\Fl(n)^s$
together with a (diagonal) $\SLL(n)$ linearization. But $\Pic(\Fl(n)^s)=\Pic(\Fl(n))^s$, and every line bundle on $\Fl(n)$ has a canonical $\SLL(n)$ linearization. There is also a  unique $\SLL(n)$ linearization on any line bundle on $\Fl(n)^s$. Therefore the Picard group of $\mathcal{A}$ is the set of $s$-tuples $(\lambda_1,\dots,\lambda_s)$ of dominant fundamental weights of $\SLL(n)$.
\subsection{Picard group of $\mathcal{S}$}\label{lastthing}
Recall the definition of $\mathcal{S}$ from Definition \ref{defiSS}. Fix vector spaces $V$ and $Q$ of dimensions $r$ and $n-r$ respectively.
 Let
\begin{equation}\label{gruppe}
H=\{(A,B)\mid A\in \operatorname{GL}(V), B\in \operatorname{GL}(Q), \det A\det B=1\}
\end{equation}
 There is a natural map $\Fl(V)^s\times \Fl(Q)^s\to\mathcal{S}$ making $\mathcal{S}$ the stack quotient  $\Fl(V)^s\times \Fl(Q)^s/H$. Therefore line bundles
 on $\mathcal{S}$ are line bundles on $\Fl(V)^s\times \Fl(Q)^s$ with a $H$-linearization.

 Let $L$ be a line bundle on $\mathcal{S}$. Let the line bundle on $\Fl(V)^s\times \Fl(Q)^s$ be written as, after choosing an arbitrary lifting as a $\GL(V)^s\times \GL(Q)^s$ line bundle
 \begin{equation}\label{ams}
 L=L_{\vec{\lambda}}\boxtimes L_{\vec{\mu}}=(L_{\lambda_1}\boxtimes L_{\lambda_2}\boxtimes\dots\boxtimes L_{\lambda_s})\boxtimes (L_{\mu_1}\boxtimes  L_{\mu_2} \boxtimes \dots\boxtimes L_{\mu_s})
 \end{equation}
Therefore one gets a $H$-linearization on $L$ which differs from the pull back $H$-linearization by a character $\chi:H\to\Bbb{C}^*$. We can extend this character to $\GL(V)\times \GL(Q)$, as follows. Let $(A,B)\in \GL(V)\times \GL(Q)$. Set $\alpha=\det A\ \det B$ and let $\beta^r=\alpha$, then define
$$\chi'(A,B)= \chi(\beta^{-1}A,B).$$
 Any two choices for $\beta$ result in the same value of $\chi'(A,B)$ because $\chi$ is identity on $\SLL(V)\times\SLL(Q)$. Now characters of $\GL(V)\times \GL(Q)$ are products of determinants. Therefore we can replace $\lambda_1$ by $\lambda_1 + c(1,1,1\dots,1)$ and assume that $\chi$ extends to $\GL(V)\times \GL(W)$,  We therefore arrive at the following description of the Picard group of $\mathcal{S}$:
 \begin{lemma}\label{office}
 Let $L$ be a line bundle on $\Fl(V)^s\times \Fl(Q)^s$ with a $\GL(V)^s\times \GL(Q)^s$ equivariant structure given by \eqref{ams}. This induces a $H$ linearization and hence gives a line bundle on $\mathcal{S}$. Furthermore,
 \begin{enumerate}
 \item All line bundles on $\mathcal{S}$ arise this way.
 \item Data $(\vec{\lambda}(i),\vec{\mu}(i))$, $i=1,2$ give the same line bundle on the stack $\mathcal{S}$ if these are equal as representations of $\SL(V)^s\times \SL(Q)^s$ and
 $w(\vec{\lambda}(1),\vec{\mu}(1))=w(\vec{\lambda}(2),\vec{\mu}(2))$ where
 \begin{equation}\label{expresso}
  w(\vec{\lambda},\vec{\mu}) =(n-r)(\sum_{i=1}^s|\lambda_i|)-r(\sum_{i=1}^s|\mu_i|)
 \end{equation}

 \item A line bundle \eqref{ams} as above does not have any non-zero global sections on $\mathcal{S}$ unless the quantity \eqref{expresso} is zero,
     in which case the space of sections coincides with
    \begin{equation} \label{invariaants}
     H^0(\Fl(V)^s,L_{\lambda}\boxtimes L_{\lambda_2}\boxtimes\dots\boxtimes L_{\lambda_s})^{\SLL(V)}\tensor H^0(\Fl(Q)^s,L_{\mu_1}\boxtimes L_{\mu_2}\boxtimes\dots\boxtimes L_{\mu_s})^{\SLL(Q)}
     \end{equation}
 \end{enumerate}
\end{lemma}
\begin{proof}
 We have already shown (1), the condition in (2) is that corresponding characters on  $\Bbb{C}^*\subset H$ (here $t\in\Bbb{C}^*$ acts as multiplication by  $t^{n-r}$ on $V$, and  by $t^r$ on $Q$). The condition in (3) is that the center of $H$ should act trivially if there are non-zero invariants. Here we have used the surjection $\Bbb{C}^*\times \SLL(V)\times\SLL(Q)\twoheadrightarrow H$.
\end{proof}

\begin{proposition}\label{stbe}
\begin{enumerate}
\item[(a)]$\Pic^{+}_{\Bbb{Q}}(\mathcal{S})$  is in bijection with $\Gamma_{r,\Bbb{Q}}\times \Gamma_{n-r,\Bbb{Q}}$.
\item[(b)] $\Pic^{\deg=0}_{\Bbb{Q}}(\ms)$ is in bijection with $\Pic_{\Bbb{Q}}(\Fl(r)^s\times \Fl(n-r)^s)$.
\item[(c)]  Extremal rays of $\Pic^{+}_{\Bbb{Q}}(\mathcal{S})$ correspond to extremal rays of $\Gamma_{r,\Bbb{Q}}\times \Gamma_{n-r,\Bbb{Q}}$ which are of two kinds: Extremal rays of $\Gamma_{r,\Bbb{Q}}$ (with $(0,\dots,0)$ on the second factor of $\Gamma_{n-r,\Bbb{Q}}$), or extremal rays $\Gamma_{n-r,\Bbb{Q}}$ (with $(0,\dots,0)$ on the first factor of $\Gamma_{r,\Bbb{Q}}$).
\end{enumerate}
\end{proposition}
\begin{proof} We use Remark \ref{identify} and Proposition \ref{wellknown1}.
Let $\ml\in \Pic^{+}_{\Bbb{Q}}(\mathcal{S})$, assume that $\ml$ comes from a line bundle of the form the form \eqref{ams} which satisfies  $w(\vec{\lambda},\vec{\mu})=0$. It is easy to see that this set of $\ml$ is in bijection with all line bundles on $\Fl(V)^s/\SL(V)\times \Fl(Q)^s/\SL(Q)$, since if $(\vec{\lambda}(i),\vec{\mu}(i))$, $i=1,2$. are data which give the same $s$ representations of  $\SL(V)$, and of $\SL(Q)$ and  satisfy  $w(\vec{\lambda}(i),\vec{\mu}(i))=0,\ i=1,2$, we see using Lemma \ref{office} that $(\vec{\lambda}(i),\vec{\mu}(i))$, $i=1,2$ give isomorphic line bundles on $\mathcal{S}$.

For the reverse direction, given a point of $\Gamma_{r,\Bbb{Q}}\times \Gamma_{n-r,\Bbb{Q}}$, we assume that it corresponds to data $(\vec{\lambda},\vec{\mu})$
normalize these so that (we are working rationally, so denominators are allowed),
$$\sum_i|\lambda_i|=\sum_i|\mu_i|=0$$
Lemma \ref{office} then produces the desired line bundle on $\ms$. This proves (a). Part (b) is proved in a similar fashion. Part (c) is a consequence of (a).
\end{proof}

\section{Partial flag varieties}
\begin{defi} Let $I$ be a subset of $\{1,\dots,n\}$ or cardinality $r$.
$\operatorname{Fl}(I)$ parameterizes certain partial flags $F_{\bull}$ on $\Bbb{C}^n$: The only case $F_a$ is not defined is when the following three conditions are all satisfied $a<n$, $a\not\in I$ and $a+1\in I$. Recall that a constituent of the partial flag $F_{\bull}$ is denoted by $F_a$ where $a=\dim F_a$.
\end{defi}

\begin{remark}\label{forcing}
Line bundles on $\Fl(I)$ pullback to line bundles $L_{\lambda}$ of
$\Fl(n)$ so that $\lambda^{(b)}=\lambda^{(b-1)}$ whenever $b\in I$ and $b-1\not\in I, b>1$. Therefore
 flag varieties of the type $\Fl(I)$ provide the right setting for the study of rays of type II in $\mfq$.
\end{remark}

\begin{defi}\label{bee}
In the setting of Section \ref{rumble}, define a stack
$\mathcal{B}=(\Fl(I_1)\times\Fl(I_2)\times \dots\times \Fl(I_s))/\SL(n)$.
\end{defi}
Repeating arguments from Section \ref{repeat}, we see using Remark \ref{forcing} that
\begin{lemma} \label{list}
$\Pic(\mathcal{B)}$ is the $\Bbb{Z}$ module  formed by  triples $(\lambda_1,\dots,\lambda_s)\in \Pic(\Fl(n)^s/\SLL(n))$ of dominant weights for
$\SL(n)$ such that  for all $i\in[s]$ and $b\in I_i$ such that $b>1$ and $b-1\not\in I_i$, we have $\lambda_i^{(b)}=\lambda_{i}^{(b-1)}$.
\end{lemma}

\begin{defi}
\begin{enumerate}
\item $\Pic^{\deg=0}(\mathcal{B})$ consists of all triples $(\lambda_1,\dots,\lambda_s)$ such that $(\killing(\lambda_1), \dots,\killing(\lambda_s))$ satisfies  equality in the inequality \eqref{evineq}.
\item $\Pic^{\deg=0,+}_{\Bbb{Q}}(\mathcal{B})=\Pic^{\deg=0}_{\Bbb{Q}}(\mathcal{B})\cap \Pic^+_{\Bbb{Q}}(\mathcal{B}).$
\end{enumerate}
\end{defi}
The following is now an easy consequence of Proposition \ref{wellknown1} and Remark \ref{forcing}.
\begin{proposition}\label{secondthing}
$\Pic^{\deg=0,+}_{\Bbb{Q}}(\mathcal{B})$ is isomorphic to the cone $\mfqtwo$ by the map that takes $(\lambda_1,\dots,\lambda_s)$ to $(\killing(\lambda_1),\dots,\killing(\lambda_s))$.
\end{proposition}

\subsection{Schubert varieties, and their codimension one subvarieties}\label{codimen}
\begin{lemma}\label{plum}
Let $\widetilde{F}_{\bull}\in\Fl(n)$. All codimension one Schubert subvarieties of $\Omega_I(\widetilde{F}_{\bullet})$ can be obtained as follows. Pick $b\in I, b>1$ such that $b-1\not\in I$, let $I'=(I-\{b\})\cup\{b-1\}$.
Then, $\Omega_{I'}(\widetilde{F}_{\bullet})$ is a codimension one subvariety of $\Omega_I(\widetilde{F}_{\bullet})$ and all codimension one Schubert subvarieties arise this way.
\end{lemma}
\begin{proof}
 It is easy to check that $\Omega_{\tilde{I}}(\widetilde{F}_{\bullet})\subseteq \Omega_I(\widetilde{F}_{\bullet})$ if and only if $\tilde{i}_k\leq i_k$ for $k=1,\dots,r$
here $I=\{i_1<\dots<i_r\}$ and $\tilde{I}=\{\tilde{i}_1<\dots<\tilde{i}_r\}$ and the difference of the dimensions of these Schubert varieties is $\sum_{k=1}^r (i_k-\tilde{i}_k)$, see equation \eqref{codimI}.
The desired statement follows immediately.
\end{proof}

\begin{defi}\label{simSV}
 The definition of the closed Schubert variety, see definition \ref{closedsv}, $\Omega_I(\widetilde{F}_{\bull})$ does not involve the elements of full flags $\widetilde{F}_{\bull}$ which have been discarded in the definition of $\Fl(I)$, therefore we can define $\Omega_I(F_{\bull})$ for all $F_{\bull}\in \Fl(I)$.

Let $T(I)$ be set of ranks of the partial flags in $\Fl(I)$, i.e., $T(I)=[n]-\{a\mid a<n,a\not\in I, a+1\in I\}$.
For $F_{\bull}\in \Fl(I)$, define (with $i_0=0, i_{r+1}=n$)
$$\widehat{\Omega}^{0}_I(F_{\bull})= \{V\in\Gr(r,n)\mid \dim (V\cap F_a)=j,
\ i_{j}\leq a<i_{j+1}, a\in T(I),\ j=1,\dots,r\}$$
This is a subset of the smooth locus of the normal  projective variety  $\Omega_I({F}_{\bull})$.
 \end{defi}
\begin{lemma}
\begin{enumerate}
\item[(a)]
 The complement $\Omega_I(F_{\bull})\setminus \widehat{\Omega}^{0}_I(F_{\bull})$ is of codimension $\geq 2$ in $\Omega_I(F_{\bull})$.
\item[(b)] $\widehat{\Omega}^{0}_I(F_{\bull})$ is homogeneous for the action of the stabilizer of a fixed  partial flag $F_{\bull}\in \Fl(I).$
\end{enumerate}
\end{lemma}
\begin{proof} Part (b) follows from an easy calculation.

For (a), extend $F_{\bull}$ to a full flag $\widetilde{F}_{\bull}$. Let  $i>1\in I$ such that $i-1\not\in I$, let $I'=(I-\{i\})\cup\{i-1\}$. Then, $\Omega_{I'}(\widetilde{F}_{\bullet})$ is a codimension one subvariety of $\Omega_I(\widetilde{F}_{\bull})=\Omega_I(F_{\bull})$ and all such codimension one Schubert varieties are obtained this way. It suffices to observe that
$\Omega^0_{I'}(\widetilde{F}_{\bullet})\subset \widehat{\Omega}^{0}_I(F_{\bull})$ which is immediate from the definitions.
\end{proof}

The following lemma is { crucial} to the process of induction:
\begin{lemma}\label{induce}
Suppose $F_{\bull}\in \Fl(I)$ and $V\in \widehat{\Omega}^0_I(F_{\bull})$.
Then the partial flag $F_{\bull}$ induces full flags on $V$ and $Q=\Bbb{C}^n/Q$.
\end{lemma}
\begin{proof}
Since $F_{i_a}$ is a member of the flag and $\dim F_{i_a}\cap V=a$,
the statement for $V$ is clear. If $i_{a+1}=i_{a}+1$, then $F_{i_a}$ and $F_{i_{a+1}}$ have the same image  in $Q$. If $i_a\leq k<i_{a+1}-1$, and $a+1\leq r$ then the rank of the image of $F_k$ in $Q$ is exactly $k-a$. Therefore for $k$ in the range $i_a\leq k<i_{a+1}-1$, the rank of the image of $F_k$ in $Q$ ranges from   $i_a-a$ to $i_{a+1}-2-a$ (end points inclusive). The image of $F_{i_{a+1}}$ in $Q$ is $i_{a+1}-a-1$, and so there are no gaps in the ranks of the the images of $F_j, j\not\in T(I)$ in $Q$ (The range  $i_r\leq k\leq n$ is handled similarly).
\end{proof}
\section{The induction operation and properties}
We return to the setting of Section \ref{rumble}: Fix $(r,n,I_1,\dots,I_n)$ satisfying satisfying \eqref{dagger}.
\subsection{Definition of induction}\label{defInd}
Let $Z=\Fl(I_1)\times \Fl(I_2)\times\dots\times \Fl(I_s)$ and $\widehat{Y}\subset \Gr(r,n)\times Z$
be the universal intersection of closed Schubert varieties:
$$\widehat{Y}=\{(V,F_{\bull}(1),F_{\bull}(2),\dots,F_{\bull}(s))\mid V\in \bigcap_{i=1}^s\Omega_{I_i}(F_{\bull}(i))\}$$

The map $\pi:\widehat{Y}\to Z$ is birational and surjective. Let $\yosim\subset \widehat{Y}$ be the  universal intersection of the hatted open Schubert varieties introduced in Definition \ref{simSV},
$$\yosim=\{(V,F_{\bull}(1),\dots,F_{\bull}(s))\mid V\in \bigcap_{i=1}^s\widehat{\Omega}^0_{I_i}(F_{\bull}(i))$$

It is easy to see that $\yosim$ is smooth. Let $\aresim\subset \yosim$ be the ramification divisor of $\pi:\yosim\to Z$.
The following is immediate,
\begin{lemma}\label{yomama}
\begin{enumerate}
\item $\widehat{Y}-\yosim$ has codimension $\geq 2$ in the projective variety $\widehat{Y}$.
\item The closure of ${\pi(\aresim)}\subset Z$ has codimension $\geq 2$ in $Z$.
\item $\pi$ induces an isomorphism between $\yosim-\aresim$ and a open subset $U\subset Z$ such that all irreducible components of  $Z-U$ are of codimension $\geq 2$
in $Z$.
\item $\Pic(\yosim-\aresim)=\Pic(U)=\Pic(Z)$.
\item Any $SL(n)$-equivariant line bundle $\ml$ on  $U$ extends to a  $SL(n)$-equivariant line bundle on $Z$.
\end{enumerate}
\end{lemma}
\begin{proof}
Zariski's main theorem gives (3). For (5), we extend the line bundle $\ml$ first as a line bundle to $Z$. This extension has a canonical $SL(n)$ equivariant
structure. Now the restriction to $U$ of the extension has the same equivariant structure as $\ml$ because any two $SL(n)$ equivariant structures on a line bundle on $U$ coincide (use the codimension statement in (3)).
\end{proof}

There is a natural map of stacks  $p:\yosim/\SLL(n)\to \mathcal{S}$ by Lemma \ref{induce}, here $\mathcal{S}$ is the stack defined in Definition \ref{defiSS}. It carries a natural divisor $R_{\mathcal{S}}$, such
that $p^{-1}(R_{\mathcal{S}})=\aresim$. There is also a section $i:\mathcal{S}\to \yosim$ given by the direct sum construction as in Section \ref{may11proof}. We obtain a variant of the basic diagram of stacks \eqref{basicdiagram}
here (as in Definition \ref{bee}),
$\mathcal{B}=(\Fl(I_1)\times\Fl(I_2)\times \dots\times \Fl(I_s))/\SL(n)=Z/\SL(n)$, and $i'=\pi\circ i$:
\begin{equation}\label{basicdiagram1}
\xymatrix{
& \yosim/SL(n)\ar[dl]^{\pi}\ar[dr]_{p}\\
\mathcal{B}  & &     \mathcal{S}\ar@/_/[ul]_i\ar[ll]^{i'}}
\end{equation}
\begin{defi}
Let $\ml\in \Pic(\ms-R_{\ms})$. Now, $p^*\ml$ is a line bundle on $\yosim-\aresim$ which is equivariant for the action of $\operatorname{SL}(n)$. By Lemma \ref{yomama}, we get a $\SL(n)$ equivariant line bundle on $Z$, and hence one on $\Fl(n)^s$.  This defines a mapping of $\Bbb{Z}$-modules which will be called induction:
\begin{equation}\label{india}
\operatorname{Ind}: \operatorname{Pic}(\mathcal{S}-R_{\mathcal{S}})\to \Pic(\mathcal{B})\subseteq \Pic(\Fl(n)^s).
\end{equation}
\end{defi}
\subsection{Properties of Induction}

\begin{theorem}\label{shosty}
\begin{enumerate}
\item[(1)] The induction operation establishes an isomorphism between the $\Bbb{Z}$-modules $\Pic^{\deg=0}(\ms-R_{\ms})$ and $\Pic^{\deg =0}(\mathcal{B})$.
\item[(2)] For all  $\ml\in \Pic^{\deg=0}(\ms-R_{\ms})$, we have an isomorphism
$$H^0(\mathcal{S}- R_{\mathcal{S}},\mathcal{L})\leto{\sim} H^0(\mathcal{B}, \operatorname{Ind}(\mathcal{L})).$$
(Taking $\ml=\mathcal{O}$, we recover Fulton's conjecture (as in Section \ref{may11proof}) since  $\operatorname{Ind}(\mathcal{O})=\mathcal{O}$,
(given the computation of the ramification divisor as in \cite{belkaleIMRN}, see Step (a) in Section \ref{may11proof}). This proof of Fulton's conjecture is a variant of \cite{BKR}, where the argument uses a smaller partial flag variety $\mathcal{B}'$ replacing $\mathcal{B}$, i.e., a surjection $\mb\twoheadrightarrow \mb'$).
\item[(3)]  Under the bijection  in (1),
 $\Pic^+(\mathcal{S}-R_{\mathcal{S}})$ corresponds to $\Pic^{+, \deg =0}(\mathcal{B})$. Therefore
 $$\Pic^+_{\Bbb{Q}}(\mathcal{S}-R_{\mathcal{S}})=\Pic^{+, \deg =0}_{\Bbb{Q}}(\mathcal{B})=\mf_2.$$
\item[(4)]Suppose $\ml=\mathcal{O}(E)\in \Pic(\ms)$ where $E$ is a locus
with ``modular properties". Then $\operatorname{Ind}(\ml)\in \Pic(\mathcal{B})\subseteq \Pic(\Fl(n)^s)$ equals $\mathcal{O}(\widetilde{E})$, where $\widetilde{E}$ is  $p^{-1}(E)\cap (\yosim- \aresim)$ as a divisor on $U\subset Z$ (as in Lemma \ref{yomama}), hence by taking closures, on $Z$. Therefore, $\operatorname{Ind}(\ml)$ also has a modular interpretation (``off codimension $2$, the corresponding point of $\mathcal{S}$ satisfies the modular property of being in $E$").
\end{enumerate}
\end{theorem}
\begin{proof}
That $\Pic^{\deg=0}(\ms-R_{\ms})$ maps to $\Pic^{\deg =0}(\mathcal{B})$ can be seen as follows: Let $L=L_{\lambda}\boxtimes L_{\mu}\boxtimes L_{\nu}$ be the image of $\ml\in\Pic^{\deg=0}(\ms-R_{\ms})$.

Let $x=(V,F_{\bull}(1),\dots,F_{\bull}(s))\in \widehat{Y}^0-\aresim$, which can be considered an open subset of $Z$, whose complement has codimension $\geq 2$. As in Section \ref{may11proof} choose a direct sum decomposition $\Bbb{C}^n=V\oplus Q$, and let $\phi_t\in \SLL(n)$ be an automorphism which is multiplication by $t^{n-r}$ on $V$ and multiplication by $t^{-r}$ on $Q$. The point $i(p(x))=\lim_{t\to 0}\phi_t(x)$. The action of $\phi_t$ on $L_{i(p(x)}$ is the same as the action of $t\in \Bbb{C}^*$ on $\ml_p(x)$
described in Definition \ref{longy}. Therefore $\phi_t$ acts by zero on $L_{i(p(x)}$, and hence equality in the inequality \eqref{evineq} holds (see Remark \ref{markJ}), and hence $L\in \Pic^{\deg =0}(\mathcal{B})$.
Conversely, if $L=L_{\lambda_1}\boxtimes \dots\boxtimes L_{\lambda_s}$ is in
$\Pic^{\deg =0}(\mathcal{B})$, then with $x$ as above, equality in \eqref{evineq} gives a  isomorphism $L_x\to L_{i(p(x))}$, by propagating a section of $L$ at $x$ to all $\phi_t(x)$ and extending it to $t=0$ (there are no zeroes or poles of this extended section, since because we have assumed equality in  \eqref{evineq}, see \cite[Proposition 10]{BK}, also Lemma \ref{schemy} below). Therefore, $L$ is  the induction of the pull back of $L$ under $i$, which is in $\Pic^{\deg=0}(\ms-R_{\ms})$. This finishes the proof of (1).

This also shows (2), because in the above situation, the value of any section of invariant section of $L$ at $x$ is the value at $L_{i(p(x))}$
under the isomorphism  $L_x\to L_{i(p(x))}$. Therefore the given map is surjective (see Lemma \ref{schemy} below for an argument in families). It is injective because we have the section $i$.

Part (3) follows from parts (1) and (2).
\end{proof}
\begin{remark}
In fact the induction map \eqref{india} is itself an isomorphism. To show surjection, one can consider $p^*i^*\ml\tensor \ml^{-1}$ with $\ml\in\Pic(\mathcal{B})=\Pic((\yosim- \aresim)/\SLL(n))$ which can be shown to have the relevant Mumford index $0$, and proceed as in the proof of (1) above to show that $p^*i^*\ml\tensor \ml^{-1}$ is trivial.
\end{remark}

\begin{proposition}\label{bloodtest}
$\Pic^{+}_{\Bbb{Q}}(\mathcal{\ms})$ surjects onto $\Pic^{+}_{\Bbb{Q}}(\mathcal{\ms}-R_{\ms})$. This surjection has a section.
\end{proposition}
\begin{proof}
If $\ml$ is a line bundle on $\mathcal{S}-R_{\mathcal{S}}$, $\operatorname{Ind}(\ml)$ is a line bundle on $\mathcal{B}$ and we can restrict it to $\ms$. This shows that
$\Pic(\mathcal{\ms})$ surjects onto $\Pic(\mathcal{\ms}-R_{\ms})$, and there is a canonical section $\Pic(\mathcal{\ms}-R_{\ms})\to \Pic(\ms)$. This also shows $\Pic^{\deg =0}(\mathcal{\ms})$ surjects onto $\Pic^{\deg=0}(\mathcal{\ms}-R_{\ms})$ (with a section). If $\ml\in   \Pic^{\deg=0}(\mathcal{\ms}-R_{\ms})$ has non-zero global sections then $\operatorname{Ind}(L)$ restricted to $\ms$ also has a non-zero global section by Theorem \ref{shosty}, as desired.
\end{proof}
\begin{lemma}\label{schemy}
Let $\Bbb{A}^1_X=X\times \Bbb{A}^1$, where $X$ is a variety, and
$\ml$ a line bundle on $\Bbb{A}^1_X$  which is linearized for the action of $\Bbb{G}_m$ (acting on the $\Bbb{A}^1$ factor). Suppose  $\Bbb{G}_m$ acts trivially on $\ml$ restricted to $X_0=X\times 0$, Then, $\ml$ is pull back of a line bundle on $X$  via $\Bbb{A}^1_X\to X$ (with the induced $\Bbb{G}_m$ action).
\end{lemma}
\begin{proof}
If $X$ is affine, the restriction map
$$H^0(\Bbb{A}^1_X,\ml)\to H^0(X_0,\ml_{X_0})$$ is surjective on sections, hence surjective on $\Bbb{G}_m$-invariant sections as well. Lifting an invariant section (i.e., trivializing $\ml$ on $X_0$), we see that $\ml$ can be assumed to be trivial. We can then see that $H^0(\Bbb{A}^1_X,\ml)=H^0(X_0,\mathcal{O})[t]$, and therefore the lift is unique. This allows us to patch.
\end{proof}
\subsection{Proof of Theorem  1.16}\label{diracsection}

The only remaining part of Theorem 1.16 (after the proof of Theorem \ref{shosty}) is the proof of the formula \eqref{sumq1} for induction.
We start with $(y_1,\dots,y_s)\times (0,\dots ,0)$ in $\Gamma_{r,\Bbb{Q}}\times \Gamma_{n-r,\Bbb{Q}}$ and write formulas for the image of the map \ref{sibelius} in $\mf_2$.  It suffices to give formulas for induction of these, since points of the form $(0,\dots,0)\times (z_1,\dots,z_s)$ can then be treated using duality on $\SL(\Bbb{C}^n)=\SL((\Bbb{C}^n)^*)$, which is equivariant for the morphisms $\Gr(r,\Bbb{C}^n)=\Gr(n-r,(\Bbb{C}^n)^*)$. We write
$$y_1=\killing(\lambda_1), \dots, y_s=\killing(\lambda_s)$$
where $\lambda_1,\dots,\lambda_s$ are dominant fundamental weights for $\operatorname{SL}(r)$. We assume (we may need to scale to avoid denominators)
\begin{equation}\label{assumption}
\sum |\lambda_i|=0.
\end{equation}

We need therefore to induce the line bundle \eqref{ams} with $\vec{\mu}=0$. The line bundles $L_{\lambda_i}$ break up into a tensor product of line bundles by Definition \ref{line}. The stated formulas are thus reduced  to formulas for the induction of some  natural line bundles on $\mathcal{S}$: Let $\ml_a(i)$ be the line bundle whose fiber at $x=(V,Q,F'_{\bull}(1), \dots,F'_{\bull}(s), F''_{\bull}(1),\dots,F''_{\bull}(s))$ is $F'_{a}(i)/F'_{a-1}(i)$, $a>0$, $i=1\dots,s$. We need to write $\ml_a(i)$
as the pull back of a line bundle on $\mb$, or to identify the pull back of this line bundle on $\mb$ to $\Fl(n)^s$.

 Now $\pi(\yosim-\aresim)$ is an open subset $U$ of $Z$ whose complement has codimension $\geq 2$. We base change the picture to $\Fl(n)^s$ via its natural map to $Z$: Let $\widetilde{Y}^0$  (respectively $\widetilde{R}$) be the base change of $\yosim\to Z$ (respectively $\aresim$).

 Let $\widetilde{U}\subset\Fl(n)^s$ be the inverse image of $U$. Note that $\widetilde{U}\cong \widetilde{Y}^0-\widetilde{R}$, and $\widetilde{Y}^0$ is a subset of $Y$ as defined in Notation \ref{listo}, but is larger than $Y^0$ there. In fact, $Y-\widetilde{Y}^0$ is of codimension $\geq 2$ in $Y$. Let $U_Y=\widetilde{Y}^0-\widetilde{R}\cong \widetilde{U}$. A point on $U_Y=\widetilde{Y}^0- \widetilde{R}$  parameterizes certain tuples $(V, F_{\bull}(1),F_{\bull}(2),\dots,F_{\bull}(s))$ (where the flags are full). Let $\mf_b(i)$ be the vector bundle on $U_Y$ with fibers $F_{b}(i), i=1,\dots,s$.

  Consider the line bundle $p^*\ml_a(i)$ on  $\yosim-\aresim$. Our aim is to write this line bundle in terms of pull backs of some natural line bundles on $\widetilde{U}$ (which has the same Picard group as $\Pic(\Fl(n)^s)$). Let $I_i=\{i_1<\dots<i_r\}$, let
 $\ell=i_a$, we consider two cases: Proposition \ref{dirac1} follows by assembling the formulas in these cases (and Definition \ref{line})
\subsubsection{Case $i_a=i_{a-1}+1$ with $i_0=0$}
In this case $p^*\ml_a(i)\leto{\sim}\mathcal{F}_{\ell}(i)/\mathcal{F}_{\ell-1}(i)$  by the evident map $F'_{a}(i) \subset F_{\ell}(i)\cap V$. The isomorphism is because the ranks $F_{i_{a}}(i)\cap V$ are not allowed to jump in the the definition of $\widehat{\Omega}^0_{I_i}(F_{\bull})$. Clearly
$\mathcal{F}_{\ell}(i)/\mathcal{F}_{\ell-1}(i)$ is the a pull back of a line bundle from $\Fl(n)^s/\SLL(n)$, and so we have achieved our aim.
\subsubsection{Case $i_a>i_{a-1} +1$}
In this case   $p^*\ml_a$ maps to $\mathcal{F}_{\ell}(i)/\mathcal{F}_{\ell-1}(i)$ but this map on $U_Y$  has a zero on the basic divisor $\widetilde{D}$ (restricted to $U_Y$) corresponding to $j_0=i$ and $a_0=a$. The order of the zero is one because
of Lemma \ref{water} below. Therefore, $p^*\ml_a(i)$ is isomorphic to  $\mathcal{F}_{\ell}(i)/\mathcal{F}_{\ell-1}(i)(-\widetilde{D})$ on $U_Y$.Now $\mathcal{O}_{U_Y}(-\widetilde{D})$ is also pulled back from $\mathcal{O}(D)$ on $\widetilde{U}\subseteq \Fl(n)^s$ and formulas are provided in Proposition \ref{egreggy}.

\begin{remark}\label{veritas}
Consider the line bundle  on $\mathcal{S}$ whose fiber at
$x=(V,Q,F'_{\bull}(1), \dots,F'_{\bull}(s), F''_{\bull}(1),\dots,F''_{\bull}(s))$
 is $\det F'_{r}(i)=\det V$. We have $s$ different formulas for the induction of this line bundle (description is on fibers, recall $F_0=0$): Consider for $i=1,\dots,s$
 $$\det V= \tensor_{a=1}^r F'_{a}(i)/F'_{a-1}(i)$$
 These $s$ formulas therefore produce the same answer as a triple of weights for $\SLL(n)$.
\end{remark}

\subsection{Order of vanishing}
Suppose  $I=\{i_1<\dots<i_r\}\subset [n]$. Pick $a\in I$ such that $a>1$ and $a-1\not\in I$, let $I'=(I-\{a\})\cup\{a-1\}$.
Let $F_{\bull}\in F(I)$. Consider the line bundle $\ml$ on $\widehat{\Omega}^0(F_{\bull})$ whose fiber at $V\in \widehat{\Omega}^0(F_{\bull})$ is $V\cap F_a/V\cap F_{a-2}$.  Pick  $\widetilde{F}_{\bull}\in \Fl(n) $ which maps to $F_{\bull}\in Fl(I)$. Consider the constant line bundle $\mathcal{N}$ on $\widehat{\Omega}^0(F_{\bull})$ with fibers given by $\widetilde{F}_{a}/\widetilde{F}_{a-1}$.
\begin{lemma}\label{water}
The natural map $\ml\to\mathcal{N}$  on  $\widehat{\Omega}^0(F_{\bull})$ has a zero of order $1$ along $\Omega_{I'}(\widetilde{F}_{\bull})\cap  \widehat{\Omega}^0(F_{\bull})\subseteq \Gr(r,n)$.
\end{lemma}
\begin{proof}
This is follows from the functor of points description of Schubert varieties as degeneracy loci (see e.g., \cite[Appendix A]{BGHorn}).
\end{proof}

\subsection{Proof of Theorem 5.1}
We have now completed the proof of Theorem \ref{egregium} following the outline given in Section \ref{bellow}.
\subsection{Comparison with Ressayre's work}\label{compare}The induction described in Section \ref{defInd} is inspired by an induction mechanism in \cite[Section 4.1]{R1}.  Ressayre works in a very general setting (of a branching problem) similar to that of  the diagram \eqref{basicdiagram1}, but with $\Fl(n)^s/\SLL(n)$ replacing $\mb=(\Fl(I_1)\times\dots\times \Fl(I_s))/\SLL(n)$ (so Ressayre's setting is more like in the diagram \eqref{basicdiagram}). Given a line bundle $\ml$ on $\ms$, he looks for an arbitrary line bundle $\ml'$  on $\Fl(n)^s/\SLL(n)$ which restricts to $\ml$ under $i'$ (there are many ways of doing this because of the center of the group $H$ defined by \eqref{gruppe}). It seems difficult to run this extension operation with $\Fl(n)^s/\SLL(n)$ replaced by $\mb$ (because we would need to extend so that certain ``eigenvalues'' coincide). He then propagates a non-zero section of $\ml$ to $\ml'$ with possible poles; the location and multiplicity of the poles may depend upon the section chosen. The possible poles of sections  can be seen to be supported on a union of our basic divisors from Definition \ref{basico} (which produce extremal rays). More precisely, these loci can be seen to correspond to  $E_j$ considered in \cite[Section 4.1]{R1} for an optimal choice of $X^o$ (as in loc. cit.).

Our $\operatorname{Ind}(\ml)$ is produced canonically by working with partial flag varieties. We have seen that sections extend without poles, and the process is entirely explicit.
\section{Complements}
We  have not used that $\mfq$ is a facet (we have only used that it is a face, possibly $0$) of $\gammanq$. Let $q$ be the number of type I extremal rays of $\mf$.
\begin{lemma}\label{counta}
\begin{enumerate}
\item $\Bbb{Z}^q\oplus \Pic(\mathcal{B})
=\Pic(\Fl(n)^s/\SLL(n))$.
\item $\Bbb{Z}^q\oplus \Pic^{\deg =0}(\mathcal{B})
=\Pic^{\deg =0}(\Fl(n)^s/\SLL(n))$.
\end{enumerate}
\end{lemma}
\begin{proof}
Given $(\lambda_1,\dots,\lambda_s)\in \Pic(\Fl(n)^s/\SLL(n))$, we can add a multiple of the classes of type one rays of $\mfq$ to make sure that the resulting triple satisfies the conditions
in the Lemma \ref{list}. This shows (1). All type I extremal rays gives rays of $\mfq$ which satisfy equality in \eqref{evineq}. Therefore (2) follows.
\end{proof}
\begin{proposition}\label{newo}
$\Pic^{\deg =0,+}_{\Bbb{Q}}(\mathcal{B})$ spans
$\Pic^{\deg =0}_{\Bbb{Q}}(\mathcal{B})$.
\end{proposition}
\begin{proof}
By Theorem \ref{shosty}, it is sufficient to show that $\Pic^{+}_{\Bbb{Q}}(\ms-R_{\ms})$ spans $\Pic^{\deg =0}_{\Bbb{Q}}(\ms-R_{\ms})$. This is implied, by Theorem \ref{shosty} (3), by  $\Pic^{+}_{\Bbb{Q}}(\ms)$ spanning $\Pic^{\deg =0}_{\Bbb{Q}}(\ms)$, which is equivalent to $\Gamma_n(s)$ being open in $\frh_{+,n}^s$ (applied to $r$ and $n-r$).
This is well known (see the discussion following Proposition 7 in \cite{FBulletin}).
\end{proof}
Proposition \ref{newo} proves (using Proposition \ref{secondthing}, and Lemma \ref{counta}) that $\mfq=\Bbb{Q}_{\geq 0}^q\oplus \mfqtwo$ is a codimension one face of $\gammanq$ (i.e., a facet) reproving the result of \cite{KTW}.

Recall that $\ml\in \Pic(\ms)$ is a line bundle on  $\Fl(V)^s\times \Fl(Q)^s$ which is equivariant for the action of the group $H$ (defined in \eqref{gruppe}) on $\Fl(V)^s\times \Fl(Q)^s$. The irreducible components of $R$ on $\Fl(V)^s\times \Fl(Q)^s$ (the inverse image of $R_{\ms}$) are invariant under the action of $H$ (see Remark \ref{remmy} below), therefore if these irreducible components are listed as $R_1,\dots,R_c$,  we obtain line bundles $\mathcal{O}(R_i)$ which are all $H$ linearized. Clearly these line bundles  lie in the kernel of the map
$\Pic(\ms)\to \Pic(\ms-R_{\ms})$. In fact, they give a basis:
\begin{proposition}\label{bee131}
\begin{enumerate}
\item $\mathcal{O}(R_i), i=1,\dots,c$ give a $\Bbb{Z}$-basis for the kernel of  $\Pic(\ms)\to \Pic(\ms-R_{\ms})$.
\item $\mathcal{O}(R_i),i=1,\dots,c$ give extremal rays of $\Gamma_{r,\Bbb{Q}}\times \Gamma_{n-r,\Bbb{Q}}\cong \Pic^+_{\Bbb{Q}}(\ms)$ (see
Proposition \ref{stbe}).
\item An extremal ray of $\Gamma_{r,\Bbb{Q}}\times \Gamma_{n-r,\Bbb{Q}}\cong \Pic^+_{\Bbb{Q}}(\ms)$ is generated by some  $\mathcal{O}(R_i)$ if and only if it inducts to zero under the induction map \eqref{sibelius}
\end{enumerate}
\end{proposition}
\begin{proof}
 We show first that they span in (1): If a line bundle on $\ms$ has a section on $\ms-R_{\ms}$, then the $H$ equivariant line bundle $\ml$ on $\Fl(V)^s\times \Fl(Q)^s$  has a section over $\Fl(V)^s\times \Fl(Q)^s-R$. Therefore the line bundle is isomorphic to $\mathcal{O}(\sum m_i R_i)$ which comes with a $H$-linearization. We show that this linearization agrees
with the one we started with on $\ml$. This is true because both have sections on $\mathcal{S}-R_{\ms}$ and hence satisfy the  $\deg=0$ condition (see Section \ref{lastthing}).

We show that they are linearly independent in (1). If
$\ml=\mathcal{O}(\sum_{i\in A} m_iR_i)= \mathcal{O}(\sum_{j\in B} n_jR_j)$ as $H$ equivariant line bundles, then they are equal also as $\SL(V)\times \SL(Q)$ equivariant line bundles. Here $m_i$ and $n_j$ are positive integers and $A,B$ disjoint non-empty subsets of $\{1,\dots,c\}$.
But the description of $\ml$ shows that it has two linearly independent invariant sections, and hence $\mathcal{O}(mR)$ has at least two linearly independent invariant sections for large $m$ contradicting Fulton's conjecture using the identification
of the line bundle $\mathcal{O}(R)$ in \cite{belkaleIMRN} (see Definition \ref{seeB}, and the last part of the Proof of Theorem \ref{may11} in Section \ref{may11proof}).

(2) follows from the method of proof of Lemma \ref{Fproperty}, using
the fact that $\mathcal{O}(mR_i)$ has exactly one invariant section for all all $m\geq 0$ (since $R_i$ is an irreducible component of $R$ which has this property). Theorem \ref{shosty}(3) gives the third conclusion, note that $\mathcal{O}(\sum m_i R_i)$ has an invariant section if and only if $m_i\geq 0$ (because if say $m_1,\dots,m_s$ are negative, and the rest non-negative  $\mathcal{O}(\sum_{i>s}m_iR_i)$ will have at least two linearly independent  invariant sections).
\end{proof}
\begin{corollary}
 Let $q$ be the number of type I extremal rays of $\mfq$, and $c$ the number of connected components of ${R}_{\ms}$. Then $c=q-s+1$.
\end{corollary}
\begin{proof} We count dimensions in Lemma \ref{counta}(2) and use
$\dim\Pic^{\deg =0}_{\Bbb{Q}}(\mathcal{B})=\dim\Pic^{\deg= 0}_{\Bbb{Q}}(\ms-R_{\ms})= \dim\Pic^{\deg= 0}_{\Bbb{Q}}(\ms)-c$ which we calculate to be equal to $s(r-1+n-r-1)-c=3n-2s-c$,
while $\dim\Pic^{\deg =0}_{\Bbb{Q}}(\Fl(n)^s) =s(n-1)-1$.
\end{proof}
\begin{remark}\label{remmy}
If a connected group acts on an algebraic variety $X$ and $R$ is a $G$-stable divisor, then every irreducible component of
$R$ is also stable under $G$ (Proof: Delete all pairwise intersections of irreducible components of $R$ from $R$ and consider the induced action of $G$ on this variety).
\end{remark}
\begin{remark}\label{stingray}
Call a ray $\Bbb{Q}_{\geq 0}(\killing(\lambda_1),\dots,\killing(\lambda_s))\subseteq\gammanq$ a F-ray if the rank of $H^0(\Fl(n)^s,  L_{m\lambda_1}\boxtimes L_{m\lambda_2}\boxtimes\dots\boxtimes L_{m\lambda_s})^{SL(n)}$ is one for all sufficiently divisible $m$. Theorem
\ref{one}, shows that type I extremal rays (i.e., our basic extremal rays of $\gammanq$) of $\mfq$ are F-rays. The induction of
a F-ray need not necessarily be a F-ray. This is because if $\ml\in\Pic^{+}_{\Bbb{Q}}(\mathcal{S})$, then the pull back of $\operatorname{Ind}(\ml)$ under $i'$ only agrees with $\ml$ outside of $R_{\mathcal{S}}$. Therefore the pull back is $\ml'(R')$ such that
$R'$ is a Cartier divisor (possibly negative, positive or zero) on $\mathcal{S}$ supported on $R_{\mathcal{S}}$. Therefore not all extremal rays of $\gammanq$ need to be F-rays (see Section \ref{Reg} for an example).
\end{remark}

\section{Examples}\label{tonne}
We examine some examples when $s=3$ for various values of $n$.
\subsection{For $n=2$}
 We have only one type of facet $\mf$ up to symmetry: Given by $r=1$, $I_1=\{1\}, I_2=I_3=\{2\}$. The basic extremal rays produced correspond to $j_0=2$ or $j_0=3$, and $a_0=2$. These produce the two extremal rays (using identifications $\killing$), and Proposition \ref{egreggy} $\Bbb{Q}_{\geq 0}(\omega_1,\omega_1,0)$, and $\Bbb{Q}_{\geq 0}(\omega_1,0,\omega_1)$. Therefore all in all, we get $3$ extremal rays, and each one lies on two regular facets. There is no induction in this example because $\Gamma_1(s)$ is trivial.
\subsection{For $n=3$}
We first list the basic extremal rays for $r=1$: The facet correspond to $I_1=\{1\}, I_2=I_3=\{3\}$. The basic extremal rays produced correspond to $j_0=2$ or $j_0=3$, and $a_0=2$. These two extremal rays are given by
$\Bbb{Q}_{\geq 0}(\omega_1,\omega_2,0)$, and $\Bbb{Q}_{\geq 0}(\omega_1,0,\omega_2)$. We get $6$ rays of this form.

The remaining choice for extremal rays for  $r=1$ is $I_1=\{3\}, I_2=I_3=\{2\}$, there are three choices for $(j_0,a_0)$ now and we get the extremal rays $\Bbb{Q}_{\geq 0}(\omega_2,\omega_2,\omega_2)$, and $\Bbb{Q}_{\geq 0}(\omega_1,\omega_2,0)$ and $\Bbb{Q}_{\geq 0}(\omega_1,0,\omega_2)$.

Altogether we have produced extremal rays $\Bbb{Q}_{\geq 0}(\omega_2,\omega_2,\omega_2)$, and six permutations of
$\Bbb{Q}_{\geq 0}(\omega_1,0,\omega_2)$.

To get the basic extremal rays from $r=2$, we just dualize the above weights (using Grassmann duality), and just get one new extremal ray  $\Bbb{Q}_{\geq 0}(\omega_2,\omega_2,\omega_2)$.

Doing the induction operation on $r=2$, we have two choices of faces (up to permutations):

\begin{enumerate}
\item $I_1=\{1,2\}$ and $I_2=I_3=\{2,3\}$. Inducing $(0,\omega_1,\omega_1)$ we get the ray $\Bbb{Q}_{\geq 0}(\omega_2,\omega_2,\omega_2)$. Therefore we have an example of an extremal ray which is type I on one face and type II on another. Inducing $(\omega_1,\omega_1,0)$, we get the ray $\Bbb{Q}_{\geq 0}(\omega_1,\omega_2,0)$.
\item $I_1=I_2=\{1,3\}$,$I_3=\{2,3\}$. Again, induction does not produce any new extremal rays.
\end{enumerate}

\subsection{For $n=4$}
The calculation work out the same way, with one surprise: We get one extremal ray which is not a triple of dominant fundamental weights, the triple $\Bbb{Q}_{\geq 0}(\omega_1+\omega_2,\omega_2,\omega_2)$ as in Example \ref{thesis}.

\subsection{Higher $n$}
Example 7.13 in \cite[Section 7]{DW}, gives an example of an extremal ray for $n=8$ which needs to be induced, since it does not have the Fulton scaling property. By \cite[Section 7]{DW}, one knows  that all extremal rays for $n\leq 7$ have the Fulton type scaling property. It is not clear if this is the first example of a extremal ray that is produced only by induction.

\subsection{An example in $\SL(9)$}\label{Reg}
The following example for $s=3$ was communicated to the author by Ressayre.  Let  (the data below is in $\Pic_{\Bbb{Q}}(\Fl(9)^3)$, see Remark \ref{identify},
$$\lambda_1=(3,3,3,2,2,2,2,1,0), \lambda_2=\lambda_3=(2,2,2,1,1,1,0,0,0).$$

Then $\Bbb{Q}_{\geq 0}(\killing(\lambda_1),\killing(\lambda_2),\killing(\lambda_3))$ is an extremal ray of $\Gamma_{9,\Bbb{Q}}(3)$, the actual example communicated to the author was the  triple of  duals of these representations, but duals of extremal rays are extremal. It was also shown by Ressayre that this was not a F-ray (i.e., the rank of the corresponding space of invariants is one, see Remark \ref{stingray}).

This ray can be shown to be on the facet $\mfq$  (the search for $I_1$, $I_2$, $I_3$ was made using some ideas from \cite{BTIFR}) given by the intersection number one situation provided by  $I_1=\{3,7,8\}$, $I_2=I_3=\{3,6,9\}$  in $\Gr(3,9)$. The following are easy to check by formulas provided in previous sections:
\begin{enumerate}
\item The representations $\lambda'_1=\lambda'_2=\lambda'_3=(1,1,0)$ give an extremal ray of $\Gamma_{3,\Bbb{Q}}(3)$. This is a type I extremal ray of the facet corresponding to $I_1=\{3\}$, $I_2=I_3=\{2\}$ in $\Gr(1,3)$ with $j_0=1$ and $T=\{2\}$.
In particular this is an F-extremal ray (see Remark \ref{stingray}) of $\Gamma_{3,\Bbb{Q}}(3)$.
\item The induction of this triple (with trivial representations on $\Gamma_{6,\Bbb{Q}}(3)$) gives $(\lambda_1,\lambda_2,\lambda_3)$ by a long but straightforward calculation by hand using the formulas for induction and the Littlewood-Richardson rule. The corresponding line bundle on $\Fl(9)^3$ is of the form $\mathcal{O}(\widetilde{E})$ for a divisor $\widetilde{E}$ which has a modular interpretation generically, see Theorem \ref{shosty}, and has hence a canonical invariant section (the space of invariants is two dimensional).
\end{enumerate}

\vspace{0.05 in}

\noindent
Department of Mathematics, University of North Carolina, Chapel Hill, NC 27599\\
{{email: belkale@email.unc.edu}}
\vspace{0.05 in}

\end{document}